\documentclass[a4paper, 10pt, DIV=10, headinclude=false, footinclude=false]{scrartcl}

\usepackage{amsmath, amsthm, amssymb, amsfonts, esint}
\usepackage[utf8]{inputenc}
\usepackage[english]{babel}
\usepackage{url}
\usepackage{mathtools}
\usepackage[hidelinks]{hyperref}
\usepackage{cleveref}

\numberwithin{figure}{section}
\numberwithin{table}{section}
\numberwithin{equation}{section}

\newenvironment{abstr}[1]{ \vspace{.05in}\footnotesize
	\parindent .2in
	{\upshape\bfseries #1. }\ignorespaces}{\par\vspace{.1in}}
\newenvironment{Abstract}{\begin{abstr}{Abstract}}{\end{abstr}}
\newenvironment{keywords}{\begin{abstr}{Key words}}{\end{abstr}}
\newenvironment{AMS}{\begin{abstr}{AMS subject classifications}}{\end{abstr}}

\newtheorem{theorem}{Theorem}[section]
\newtheorem{lemma}[theorem]{Lemma}
\newtheorem{corollary}[theorem]{Corollary}

\theoremstyle{definition}
\newtheorem{definition}[theorem]{Definition}

\newtheorem{remark}[theorem]{Remark}

\crefname{theorem}{Theorem}{Theorems}
\crefname{lemma}{Lemma}{Lemmas}
\crefname{remark}{Remark}{Remarks}
\crefname{corollary}{Corollary}{Corollaries}

\newcommand{\eff}{0}
\newcommand{\HMM}{\mathrm{H}}

\newcommand{\firstindex}{i}
\newcommand{\secondindex}{j}

\newcommand{\Lp}[2]{\mathrm{L}^{#1}\left(#2\right)}
\newcommand{\Hk}[2]{\mathrm{H}^{#1}\left(#2\right)}
\newcommand{\Hper}[1]{\mathrm{H}^1_\#\left(#1\right)}
\newcommand{\Hcurl}[1]{\mathrm{H}\left(\curl,#1\right)}
\newcommand{\Hcurldirichlet}[1]{\mathrm{H}_{0}\left(\curl,#1\right)}
\newcommand{\Ck}[2]{\mathrm{C}^{#1}\left(#2\right)}
\newcommand{\Wkp}[3]{\mathrm{W}^{#1,#2}\left(#3\right)}

\newcommand{\unitcell}{Y}
\newcommand{\samplingdomainname}[1]{\unitcell^{#1}}
\newcommand{\samplingdomain}[2]{\unitcell^{#1}\left(#2\right)}

\newcommand{\spaceVmac}{{\mathrm{V}^{\mrm{mac}}}}
\newcommand{\spaceVmic}{{\mathrm{V}^{\mrm{mic}}}}
\newcommand{\spaceVH}{\mathrm{V}_{H}}
\newcommand{\spaceVh}{\mrm{V}^{h}}
\newcommand{\spaceX}{\mrm{X}}

\newcommand{\intt}[2]{\int_{0}^{t}#1 \D #2}

\newcommand{\identityoperator}{\operatorname{id}}

\newcommand{\norm}[2]{\left\lVert#1\right\rVert_{#2}}
\newcommand{\seminorm}[2]{\left\lvert#1\right\rvert_{#2}}
\newcommand{\linearform}[2]{#1 \! \left(#2\right)}
\newcommand{\bilinearform}[3]{#1 \! \left(#2,#3\right)}
\newcommand{\namemicrobilinearform}[1]{s_{#1}}
\newcommand{\microkappabilinearform}[3]{\bilinearform{s_{#1}^{\delta}}{#2}{#3}}
\newcommand{\discretemicrobilinearform}[3]{\bilinearform{s_{#1}^{h}}{#2}{#3}}
\newcommand{\remainder}[2]{\Lambda\!\left(#1,#2\right)}

\newcommand*{\wellposedness}{well-posedness }
\newcommand*{\wellposed}{well-posed }
\newcommand*{\Wellposedness}{Well-posedness }

\newcommand{\R}{\mathbb{R}}
\newcommand{\N}{\mathbb{N}}

\newcommand{\operatordomain}{\mathcal{D}}
\newcommand{\computationaldomain}{\Omega}

\newcommand{\quadratureindex}{q}
\newcommand{\weight}{\gamma}

\newcommand{\coercivityconstant}{\alpha}
\newcommand{\boundednessconstant}{C_{\vec{M}}}
\newcommand{\genericconstant}{C}

\newcommand{\electricfield}{\vec{E}}
\newcommand{\delelectricfield}{\electricfield^{\delta}}
\newcommand{\effelectricfield}{\electricfield^{\eff}}

\newcommand{\magneticfield}{\vec{H}}
\newcommand{\delmagneticfield}{\magneticfield^{\delta}}
\newcommand{\effmagneticfield}{\magneticfield^{\eff}}

\newcommand{\currentdensity}{\vec{J}}

\newcommand{\polarizationfield}{\vec{P}}

\newcommand{\solutionu}{\vec{u}}
\newcommand{\delsolutionu}{\solutionu^{\delta}}
\newcommand{\effsolutionu}{\solutionu^{\eff}}

\newcommand{\HMMsolutionu}{\solutionu^{\HMM}}

\newcommand{\mcal}[1]{\mathcal{#1}}
\newcommand{\mbb}[1]{\mathbb{#1}}
\newcommand{\mrm}[1]{\mathrm{#1}}

\newcommand{\eps}{\varepsilon}

\newcommand{\effeps}{\eps^{\eff}}
\newcommand{\deleps}{\eps^{\delta}}

\newcommand{\effmu}{\mu^{\eff}}
\newcommand{\delmu}{\mu^{\delta}}

\newcommand{\effsigma}{\sigma^{\eff}}
\newcommand{\delsigma}{\sigma^{\delta}}

\newcommand{\D}{\: \mathrm{d}}
\newcommand{\grady}{\nabla_{y}}

\newcommand{\with}{\,:\,}

\newcommand{\euler}{\mathrm{e}}

\newcommand{\ek}{\vec{e}_{\firstindex}}
\newcommand{\el}{\vec{e}_{\secondindex}}

\newcommand{\wM}{w^{\mrm{M}}}
\newcommand{\wkM}{\wM_{\firstindex}}
\newcommand{\wlM}{\wM_{\secondindex}}
\newcommand{\wMh}{w^{\mrm{M},h}}
\newcommand{\wkMh}{\wMh_{\firstindex}}
\newcommand{\wlMh}{\wMh_{\secondindex}}

\newcommand{\wG}{\overline{w}}
\newcommand{\wkG}{\wG_{\firstindex}}
\newcommand{\wlG}{\wG_{\secondindex}}
\newcommand{\wGh}{\wG^{h}}
\newcommand{\wkGh}{\wGh_{\firstindex}}
\newcommand{\wlGh}{\wGh_{\secondindex}}

\newcommand{\wN}{w^{0}}

\newcommand{\wlN}{\wN_{\secondindex}}
\newcommand{\wNh}{w^{0,h}}

\newcommand{\wlNh}{\wNh_{\secondindex}}

\newcommand{\dimMaxwell}{n}
\newcommand{\dimMicro}{N}

\newcommand{\eqdot}{\, .}
\newcommand{\eqcomma}{\, ,}

\newcommand{\Nedelec}[2]{\mrm{V}^{#1}\left(\curl,#2\right)}
\newcommand{\DirichletNedelec}[2]{\mrm{V}^{#1}_{0}\left(\curl,#2\right)}

\newcommand{\curl}{\operatorname{curl}}
\renewcommand{\div}{\operatorname{div}}

\def\vec#1{\ensuremath{\mathchoice
		{\mbox{\boldmath$\displaystyle\mathbf{#1}$}}
		{\mbox{\boldmath$\textstyle\mathbf{#1}$}}
		{\mbox{\boldmath$\scriptstyle\mathbf{#1}$}}
		{\mbox{\boldmath$\scriptscriptstyle\mathbf{#1}$}}}}

\allowdisplaybreaks[4]

\begin{document}
	
	\title{The Heterogeneous Multiscale Method for dispersive Maxwell systems%
		\thanks{The work conducted at KIT was funded by the Deutsche Forschungsgemeinschaft (DFG, German Research Foundation) -- Project-ID 258734477 -- SFB 1173}
	}
	\author{\sffamily Philip Freese\footnotemark[2]}
	\date{}
	\maketitle
	
	\renewcommand{\thefootnote}{\fnsymbol{footnote}}
	\footnotetext[2]{Department of Mathematics, University of Augsburg, Universitätsstr. 14, 86159 Augsburg, Germany, \\\href{mailto:philip.freese@uni-a.de}{philip.freese@uni-a.de}}
	\renewcommand{\thefootnote}{\arabic{footnote}}
	\vspace{-2em}
	\begin{Abstract}
		In this work, we apply the finite element heterogeneous multiscale method to a class of dispersive first-order time-dependent Maxwell systems. For this purpose, we use an analytic homogenization result, which shows that the effective system contains additional dispersive effects. We provide a careful study of the (time-dependent) micro problems, including $\mrm{H}^2$ and micro errors estimates. Eventually, we prove a semi-discrete error estimate for the method.
	\end{Abstract}

	

	\section{Introduction}
	The simulation of electromagnetic wave propagation in heterogeneous structures has been and continues to be a challenging task. Especially, the class of \textit{metamaterials}, first constructed in \cite{Smith2000}, attracts high interest due to their extraordinary material properties such as negative refraction, perfect lensing or electromagnetic cloaking \cite{Pendry2000,Shelby2001,LiHuang2013}. Those artificial materials consist of so-called \textit{unit cells} whose size, denoted by $\delta>0$, is significantly smaller than the wavelength of an incident wave. While the propagation of electromagnetic waves in general is modeled using Maxwell's equations, heterogeneous materials enter these equations via rapidly varying coefficients. In the simplest case covered by our analysis, this means that the electric permittivity $\deleps$, the magnetic permeability $\delmu$ and the electric conductivity $\delsigma$ are highly oscillatory, which is indicated by the superscript $\delta$. Therefore, also the electric field $\delelectricfield$ and the magnetic field $\delmagneticfield$ depend on the characteristic size of the unit cells. Given the current density $\currentdensity^{\delta} = \sigma^{\delta} \delelectricfield + \currentdensity^{\delta}_\mrm{ext}$, the goal is to find the electromagnetic field that is the solution of:
	\begin{subequations}\label{sys:Maxwellconductivity_het}
		\begin{align}
		&\deleps(x) \partial_{t} \delelectricfield(t,x) + \delsigma(x) \delelectricfield(t,x)  - \curl \delmagneticfield(t,x)  = -\currentdensity^{\delta}_\mrm{ext}(t,x)  \eqcomma \\
		&\delmu(x) \partial_{t} \delmagneticfield(t,x)  + \curl \delelectricfield(t,x)  = \vec{0} \eqcomma
		\end{align}
	\end{subequations}
	subject to suitable initial and boundary conditions. Even with modern computer technology, the resolution of the microscopic and macroscopic scale simultaneously is not possible. Thus, a suitable multiscale method has to be applied to solve the heterogeneous Maxwell system. 
	
	The periodicity and scale separation of the materials under consideration suggests to use analytic \textit{homogenization} to derive an effective counterpart to the heterogeneous system. This homogenized set of equations describes a material that reflects the effective behavior that we are actually interested in. Except from special cases, however, the effective parameters are not known analytically. Instead, they are represented as averages over unit cells and include the solutions of partial differential equations, the so-called \textit{micro problems}, posed on the unit cells. An overview of the method of homogenization is found in \cite{Jikov1994,Cioransecu1999}. More specific for Maxwell's equations, we refer to first results in \cite{Sanchez1980} and \cite{Wellander2001} for the time-domain homogenization of the system \cref{sys:Maxwellconductivity_het}. The resulting effective Maxwell system consists of parameters $\effeps$, $\effmu$ and $\effsigma$ which are the homogenized variants of the permittivity, permeability and conductivity. 
	Here, the superscript $^{\eff}$ indicates that these quantities result from a suitable limit for $\delta\to 0$. 
	The effective electric and magnetic field $\effelectricfield$, $\effmagneticfield$ are solution of the system:
	\begin{subequations}
		\begin{align*}
		&\effeps \partial_{t} \effelectricfield(t,x)  + \effsigma \effelectricfield(t,x)  + \int_0^t \chi^{\eff}(t-s) \effelectricfield(s,x)  \D s - \curl \effmagneticfield(t,x)  =  -\currentdensity^{\eff}_\mrm{ext}(t,x)  \eqcomma \\
		&\effmu \partial_{t} \effmagneticfield(t,x)  + \curl \effelectricfield(t,x)  = \vec{0} \eqdot
		\end{align*}
	\end{subequations}
	The additional parameter $\chi^{\eff}$ is a non-local material law, which is solely present due to the damping of the conductivity. Precisely this expression, describing dispersive effects, is associated with the experimentally observed properties of metamaterials. This is consistent with physically derived models for metamaterials, such as the Drude \cite{Pendry1996} or the Lorentz model \cite{Smith2000Lorentz}. The rigorous derivation of the effective system is based on the concept of two-scale convergence \cite{Nguetseng1989,Allaire1992}. We follow the approach from \cite{Bokil2018} that includes linear dispersive effects and is thus valid for a far more general class of Maxwell systems. In \cite{Bossavit2005} even time-dependent parameters are considered. The latter references use the method of periodic unfolding \cite{Cioranescu2002,Cioranescu2008}, a generalization of two-scale convergence. 
	
	In this work, we aim to approximate the effective solution using the finite element method. Since no explicit representation of the effective parameters is available, a direct application of the FEM is not possible. Thus, we also use the finite element scheme to approximate the effective parameters. This combination of a macroscopic and a microscopic solver fits in the framework of the Heterogeneous Multiscale Method (HMM) \cite{Engquist2003,AbdulleEEngquist2012,Abdulle2009,E2005} which provides a general framework for the solution of multiscale problems. The HMM applied to the Maxwell system in time harmonic formulation is presented in \cite{Ciarlet2017,Henning2016}. An application to the time-domain Maxwell system without any damping is considered in \cite{Hochbruck2018}. As alternative methods to cover the multiscale character, we mention the Multiscale Finite Element Method \cite{Hou1997,Efendiev2009} and the multiscale hybrid-mixed finite element method from \cite{Harder2013}. The latter one has recently been applied to an instantaneous heterogeneous Maxwell system in \cite{Lanteri2018}. Moreover, multiscale methods for the Maxwell's equations with memory effects are presented in \cite{Zhang2010}. Finally, we mention the Localized Orthogonal Decomposition \cite{Malqvist2014,Henning2014,Peterseim2016}. First introduced for the Laplace operator, it was shown in \cite{Gallistl2018} that this technique is also applicable to the (stationary) Maxwell system. The aforementioned multiscale methods are especially useful for unstructured heterogeneities. Here, however, we consider (locally) periodic structures and thus, the HMM is an appropriate choice. 
	
	Our main contribution is the semi-discrete error estimate for the application of the HMM to an effective Maxwell system including dispersive effects. We extend techniques from \cite{Hochbruck2018,Hipp2017} to the integro-differential structure of the effective problem. Moreover, we provide new micro error estimates that are crucial in the analysis of the proposed method. For these estimates, a $\mrm{H}^{2}$ result for the so-called Sobolev equation is mandatory, which we also prove.
	
	The structure of this paper is as follows. In \cref{sec:The heterogeneous Maxwell system}, we introduce general damped Maxwell systems and our main assumptions. We recapture a homogenization result in \cref{sec:Homogenization of Maxwells equations in locally periodic media} and present a reformulation of the effective Maxwell system. Further, in \cref{sec:Wellposedness of microscopic and macroscopic problems}, we investigate the \wellposedness and stability of the microscopic and macroscopic problems that arise from homogenization and derive the $\mrm{H}^{2}$ estimate. In \cref{sec:The Finite Element Heterogeneous Multiscale Method} we apply the finite element heterogeneous multiscale method to the effective Maxwell system. Eventually, \cref{sec:Semi-discrete error analysis} is concerned with the semi-discrete error analysis, including the examination of the micro errors.

	\subsection{Notation}
	Throughout, we use a generic constant $\genericconstant>0$, which can have different values at each appearance. Let $\Omega\subseteq \R^{3}$, referred to as the macroscopic domain, be open, bounded and simply connected. Furthermore, $T$ denotes the final time of interest, and we introduce the (microscopic) unit cell $\unitcell=\left(0,1\right)^{3}$. We abbreviate the mean as $\fint_{\unitcell} \cdot \D y = \frac{1}{\seminorm{\unitcell}{}}\int_{\unitcell} \cdot \D y$. Moreover, with $\el \in \R^{d}$, for $\secondindex \in \{1,\dots,d\}$ we denote the $\secondindex$-th canonical basis vector of $\R^{d}$. A bold $\vec{0}$ represents a vector or a matrix with all entries equal to $0$. By $\Lp{2}{\Omega}$ we denote the space of square integrable functions over $\Omega$, which is equipped with the inner product $\bilinearform{}{\cdot}{\cdot}_{\Lp{2}{\Omega;\R^{d}}} = \bilinearform{}{\cdot}{\cdot}$ and norm $\norm{\cdot}{\Lp{2}{\computationaldomain;\R^{d}}} = \norm{\cdot}{}$. The Sobolev space of $\mrm{L}^{2}$-functions with $k$-th weak derivative in $\mrm{L}^{2}$ is denoted by  $\Hk{k}{\Omega}$. Additionally, the space $\Hper{\unitcell}$ contains all functions in $\Hk{1}{\unitcell}$ that are periodic and have zero mean value. Furthermore, the space $\Hcurl{\Omega}$ consists of the $\mrm{L}^{2}$-functions whose weak curl is also in $\Lp{2}{\Omega}$. Finally, to incorporate perfectly conducting boundary conditions, we introduce $\Hcurldirichlet{\Omega}$, which is the closure of $\mrm{C}_{0}^{\infty}\left(\computationaldomain;\R^{3}\right)$ with respect to the $\Hcurl{\Omega}$-norm.

	\section{The heterogeneous Maxwell system}\label{sec:The heterogeneous Maxwell system}
	Throughout, we denote by $\delta>0$ the characteristic length scale of microscopic oscillations. The parameters in the Maxwell system are assumed to be locally periodic.
	\begin{definition}[Locally periodic parameter]\label{def:locally periodic parameter}
		Let $\delta>0$, $n\in\N$. A tensor $\alpha^{\delta}:\computationaldomain \rightarrow \R^{\dimMaxwell\times \dimMaxwell}$ is called locally $\delta$-periodic if there exists a tensor $\alpha: \computationaldomain \times\R^{3} \rightarrow \R^{\dimMaxwell\times \dimMaxwell}$, which is $\unitcell$-periodic in its second argument and $\alpha^{\delta}(x) = \alpha\left(x, \frac{x}{\delta}\right)$ holds for almost every $x\in\computationaldomain$. 
	\end{definition}
	The example \eqref{sys:Maxwellconductivity_het} fits in a rather general class of (damped) Maxwell systems that describe a broad class of linear dispersive materials. The dispersion, representing frequency dependent material response, enters Maxwell's equations via the polarization and magnetization. Here, we only consider electric effects, and thus assume that no magnetization is present. This is just for convenience, and we point out that all results below hold true for magnetic effects as well. We consider the polarization $\polarizationfield^{\delta}$ to be given as the solution of an $N_{E}$-th order ODE, $N_{E}\geq 0$, which can be written as a system of first order ODE's. Therefore, introduce the collection of all time derivatives of $\polarizationfield^{\delta}$ up to order $N_{E}-1$ as
	\begin{align*}
	\mbb{P}^{\delta}(t,x) = \begin{pmatrix}
	\polarizationfield^{\delta}(t,x)^{T} & \partial_{t} \polarizationfield^{\delta}(t,x)^{T} & \dots & \partial_{t}^{N_{E} - 1} \polarizationfield^{\delta}(t,x)^{T} 
	\end{pmatrix}^{T} \in\R^{3N_{E}} \eqdot
	\end{align*}
	The polarizations that fit in our model are given as solutions of
	\begin{align*}
	\vec{M}^{\delta}_{\mbb{P}}(x) \partial_{t} \mbb{P}^{\delta}(t,x) + \vec{R}^{\delta}_{\mbb{PP}}(x) \mbb{P}^{\delta}(t,x) + \vec{R}^{\delta}_{\mbb{P}\mrm{E}}(x) \delelectricfield(t,x) = \vec{0}_{3N_{E}} \eqcomma \quad \mbb{P}^{\delta}(0,x) = \vec{0}_{3 N_{E}} \eqcomma
	\end{align*}
	where the dimensions of the matrices are $\vec{M}^{\delta}_{\mbb{P}}$, $\vec{R}^{\delta}_{\mbb{PP}}\in\R^{3N_{E}\times 3N_{E}}$ and $\vec{R}^{\delta}_{\mbb{P}\mrm{E}} \in \R^{3N_{E} \times 3}$. Examples that fit in this framework are the Debye model for orientation polarization or the Lorentz model, see \cite[Section 6]{Bokil2018}. Now, we introduce the abstract solution
	\begin{align*}
	\delsolutionu(t,x) = \begin{pmatrix}
	\delelectricfield(t,x) \\ \mbb{P}^{\delta}(t,x) \\ \delmagneticfield(t,x)
	\end{pmatrix} \in \R^{3(2+N_{E})}
	\eqdot
	\end{align*}
	and the parameter matrices 
	\begin{align*}
	\vec{M}^{\delta}(x) = \begin{pmatrix}
	\deleps(x) & & \\ & \vec{M}^{\delta}_{\mbb{P}}(x) & \\ & & \delmu(x)
	\end{pmatrix} \eqcomma \quad
	\vec{R}^{\delta}(x) = \begin{pmatrix}
	\vec{R}^{\delta}_{\mrm{EE}}(x) & \vec{R}^{\delta}_{\mrm{E}\mbb{P}}(x) & \\ \vec{R}^{\delta}_{\mbb{P}\mrm{E}}(x) & \vec{R}^{\delta}_{\mbb{PP}}(x) & \\
	& & \vec{0} 
	\end{pmatrix} \eqdot
	\end{align*}
	Additionally, we denote the Maxwell operator and the right-hand side as
	\begin{align*}
	\vec{A}\delsolutionu(t,x) = \begin{pmatrix}
	-\curl \delelectricfield(t,x) \\
	\vec{0} \\
	\curl \delmagneticfield(t,x)
	\end{pmatrix}\eqcomma \quad \vec{g}^{\delta}(t,x) &= \begin{pmatrix}
	-\currentdensity^{\delta}_\mrm{ext}(t,x) \\ \vec{0} \\ \vec{0}
	\end{pmatrix}\eqdot
	\end{align*}
	The solution space for the Maxwell system is given as
	\begin{align}\label{def:spaceVmac}
	\spaceVmac \coloneqq \Hcurldirichlet{\computationaldomain}\times \Hcurl{\computationaldomain}^{N_{E}} \times \Hcurl{\computationaldomain} \eqdot
	\end{align}
	Hence, we obtain a family of electromagnetic fields $\vec{u}^\delta:[0,T]\to \spaceVmac$ that solves the heterogeneous evolution problem with perfectly conducting boundary condition
	\begin{subequations}\label{sys:Maxwell_Debye_ODE_heterogeneous}
		\begin{align}
		\vec{M}^\delta(x)\partial_t \vec{u}^\delta(t,x) + \vec{R}^\delta(x)\vec{u}^\delta(t,x) + \vec{A}\vec{u}^\delta(t,x) &= \vec{g}^\delta(t,x) &&\text{in }(0,T)\times\Omega \eqcomma \label{eq:eqsys:Maxwell_ODE_heterogeneous}\\
		\vec{u}^\delta(0,x) &= \vec{u}^\delta_0(x) &&\text{in }\Omega \eqcomma\\
		\vec{n}\times\vec{u}^\delta_1(t,x) &= \vec{0} &&\text{on }(0,T)\times\partial\Omega \eqdot
		\end{align}
	\end{subequations}
	We denote the dimension of this abstract Maxwell system by $\dimMaxwell\coloneqq 3 (2 + N_{E})$. For the \wellposedness of \cref{sys:Maxwell_Debye_ODE_heterogeneous} we assume that the parameter $\vec{M}^{\delta}$ is bounded and positive definite, i.e., $\vec{M}^{\delta}\in\Lp{\infty}{\computationaldomain;\R^{\dimMaxwell \times \dimMaxwell}}$ and there exist constants $\coercivityconstant,$ $\boundednessconstant>0$ such that
	\begin{align}\label{def:propertiesM}
	\coercivityconstant\lvert \vec{\Phi}\rvert^{2} \leq \vec{M}^{\delta}(x)\vec{\Phi} \cdot \vec{\Phi} \leq \boundednessconstant \lvert \vec{\Phi} \rvert^{2} \quad \text{for all }\vec{\Phi}\in\R^{\dimMaxwell} \eqdot
	\end{align}
	The damping parameter $\vec{R}^{\delta}$ is assumed to be bounded and positive semi-definite with
	\begin{align}\label{def:propertyR}
	\genericconstant_{\vec{R}} = \norm{\vec{R}}{\Lp{\infty}{\computationaldomain,\R^{\dimMaxwell\times \dimMaxwell}}} \eqdot
	\end{align}
	\Wellposedness of \cref{sys:Maxwell_Debye_ODE_heterogeneous} may be shown using semigroup theory \cite[Theorem 3.2.23]{Freese2021} or the Faedo--Galerkin approach \cite[Proposition 1]{Bossavit2005}.

	\section{Homogenization of Maxwell's equations in locally periodic media} \label{sec:Homogenization of Maxwells equations in locally periodic media}
	As explained, we are interested in the asymptotic behavior of the solution $\vec{u}^\delta$ as the periodicity length $\delta$ tends to zero. For the initial data and source, we assume
	\begin{equation}
	\vec{u}^\delta_0 \rightarrow \vec{u}_0 \quad \text{strongly in }\spaceVmac \eqcomma \quad \vec{g}^\delta \rightarrow \vec{g} \quad \text{strongly in }\Hk{1}{0,T;\Lp{2}{\Omega;\R^{\dimMaxwell}}} \eqdot \label{eq_initialcondition_source_assumption}
	\end{equation}
	We follow the approach of \cite{Bokil2018} where the periodic unfolding method \cite{Cioranescu2002, Cioranescu2008} is used to derive the effective system. In \cite[Theorem 3]{Bossavit2005}, and \cite[Theorem 5.1]{Bokil2018} it is shown that the sequence $\delsolutionu$ is uniformly bounded and converges weakly-* to a limit $\effsolutionu$ in $\Lp{\infty}{0,T;\spaceVmac}$. The following theorem states that the effective solution solves a global Maxwell problem in $(0,T)\times\Omega$. In addition, there exists a corrector $\vec{\overline{u}}$ that solves local diffusion problems posed in $(0,T)\times \unitcell$ for $x\in\Omega$.
	\begin{theorem}[{\cite[Theorem 5.2]{Bokil2018}}]\label{thm:homogeneous_system}
		For $\delta>0$, let $\vec{M}^\delta\in\Lp{\infty}{\Omega;\R^{\dimMaxwell\times \dimMaxwell}}$, symmetric and uniformly positive definite, and $\vec{R}^\delta\in\Lp{\infty}{\Omega;\R^{\dimMaxwell\times \dimMaxwell}}$, be two families of locally periodic parameters. Assume that the initial condition $\vec{u}_{0}^{\delta}$ and the source $\vec{g}^{\delta}$ satisfy \cref{eq_initialcondition_source_assumption}. Then there exists a unique effective field 
		\begin{equation*}
		\effsolutionu = \begin{pmatrix}
		{\effelectricfield} \\ 
		{\mbb{P}^{\eff}} \\
		{\effmagneticfield}
		\end{pmatrix}
		\in\Wkp{1}{\infty}{0,T;\Lp{2}{\Omega;\R^{\dimMaxwell}}}\cap \Lp{\infty}{0,T;\spaceVmac} \eqcomma
		\end{equation*}
		which solves the effective Maxwell system
		\begin{subequations}
			\begin{align*}
			\begin{split}
			\vec{M}^{\eff}(x)\partial_t \vec{u}^{\eff}(t,x) &+ \vec{\tilde{R}}^{\eff}(x) \vec{u}^{\eff}(t,x) + \partial_t \int_0^t\vec{\tilde{G}}^{\eff}(t-s,x)\vec{u}^{\eff}(s,x)\D s +\vec{A}\vec{u}^{\eff}(t,x) \\
			&= \vec{g}(t,x) - \vec{J}^{\eff}(t,x)\vec{u}_{0}(x) \quad\text{in }(0,T)\times \Omega \eqcomma
			\end{split} \\
			\vec{u}^{\eff}(0) &= \vec{u}_{0} \quad \text{ in }\Omega \eqcomma \\
			\vec{n}\times\vec{u}^{\eff}_{1} &= \vec{0}\quad \text{ on }(0,T)\times\partial\Omega \eqdot			
			\end{align*}
		\end{subequations}
		The parameters $\vec{M}^{\eff}$ and $\vec{\tilde{R}}^{\eff}$ only depend on the macroscopic (slow) variable $x$. The extra source $\vec{J}^{\eff}$ and the parameter $\vec{\tilde{G}}^{\eff}$ result from homogenization, the latter representing additional polarization effects.
	\end{theorem}
	Next, we give an equivalent formulation of the homogeneous system from \cref{thm:homogeneous_system}. The central difference is the formulation of the cell-problems as differential, rather than integral, equations and the evaluation of time derivatives. Moreover, in the spirit of the HMM, we transform the unit cell to a so-called \textit{sampling domain} $\samplingdomain{\delta}{x} = x + \delta \unitcell$, $x\in\Omega$. The measure of these sampling domains is independent of $x$ and thus denoted by $\seminorm{\samplingdomainname{\delta}}{}$. In the following result and throughout, we use the notation for the component wise gradient. For $w=\begin{pmatrix}
	w_{1} & w_{2}^{T} & w_{3} 
	\end{pmatrix}^{T} \in \R^{1 + N_{E} + 1}$ define
	$
	\nabla w \coloneqq \begin{pmatrix}
	\nabla w_{1}^{T} & \nabla w_{2}^{T} & \nabla w_{3}^{T} 
	\end{pmatrix}^{T} \in \R^{\dimMaxwell} \eqdot
	$
	\begin{corollary}\label{cor:reformulation}
		The solution $\effsolutionu$ from \cref{thm:homogeneous_system} is also solution of
		\begin{subequations}\label{sys:reformulated_effective_Maxwell_system}
			\begin{align} \label{eq:reformulated_effective_Maxwell_system_1}
			\begin{split}
			\vec{M}^{\eff}(x)\partial_t \effsolutionu(t,x) &+ \vec{R}^{\eff}(x) \effsolutionu(t,x) + \intt{\vec{G}^{\eff}(t-s,x) \effsolutionu(s,x)}{s} +\vec{A}\effsolutionu(t,x) \\
			&= \vec{g}(t,x) - \vec{J}^{\eff}(t,x)\vec{u}_{0}(x) \quad \text{in }(0,T)\times \Omega \eqcomma
			\end{split}\\
			\effsolutionu(0) &= \vec{u}_{0} \quad \text{ in }\Omega \eqcomma \\
			\vec{n}\times \effsolutionu_{1}(t,x) &= \vec{0}\quad \text{ on }(0,T)\times\partial\Omega \eqcomma
			\end{align}
		\end{subequations}
		where for $\firstindex, \secondindex = 1,\dots,\dimMaxwell$ the $\firstindex,\secondindex$-th component of the effective parameters and the extra source are given as
		\begin{subequations}\label{sys:reformulated_effective_parameters}
			\begin{align}
			(\vec{M}^{\eff}(x))_{\firstindex,\secondindex} &=
			\fint_{\samplingdomain{\delta}{x}} \vec{M}\left(x,\frac{y}{\delta}\right) \left(\el + \nabla_{y} \wlM(x,y) \right)\cdot \left(\ek + \nabla_{y} \wkM(x,y) \right) \D y \eqcomma \label{def:M_eff_symmetric_transformed} \\
			(\vec{{R}}^{\eff}(x))_{\firstindex,\secondindex} &= \fint_{\samplingdomain{\delta}{x}} \vec{R}\left(x,\frac{y}{\delta}\right)\left(\el + \grady \wlM(x,y) \right)\cdot \left(\ek + \grady \wkM(x,y) \right)\D y \eqcomma \label{def:R_eff_reformulated_transformed} \\
			\left(\vec{G}^{\eff}\left(t,x\right)\right)_{\firstindex,\secondindex} &= \fint_{\samplingdomain{\delta}{x}} \vec{R}\left(x,\frac{y}{\delta}\right) \grady \wlG\left(t,x,y\right) \cdot \left(\ek + \grady \wkM\left(x,y\right) \right)\D y \eqcomma \label{def:G_eff_reformulated_transformed}
			\end{align}
			\begin{align}
			(\vec{J}^{\eff}(t,x))_{\firstindex,\secondindex} &= \fint_{\samplingdomain{\delta}{x}} \vec{R}\left(x,\frac{y}{\delta}\right) \grady \wlN(t,x,y) \cdot \left(\ek + \grady \wkM(x,y) \right)\D y \eqdot \label{def:J_eff_reformulated_transformed}
			\end{align}
		\end{subequations}
		Let $\dimMicro\coloneqq \frac{\dimMaxwell}{3}$. The correctors $\wlM(x,\cdot)\in \Hper{\samplingdomain{\delta}{x};\R^{\dimMicro}}$ solves
		\begin{align}
		&\int_{\samplingdomain{\delta}{x}} \vec{M}\left(x,\frac{y}{\delta}\right) \left(\el + \grady \wlM(x,y)\right) \cdot \grady v(y) \D y = 0 \quad \text{for all }v\in \Hper{\samplingdomain{\delta}{x};\R^{\dimMicro}} \eqcomma \label{def:cell_prob_wM_reformulated_transformed}
		\end{align}
		The time-dependent corrector $\wlG \in \Ck{\infty}{0,T;\Hper{\samplingdomain{\delta}{x};\R^{\dimMicro}}}$ that arises due to the damping parameter $\vec{R}^{\delta}$ solves for all $t\in (0,T)$ and all $v\in\Hper{\samplingdomain{\delta}{x};\R^{\dimMicro}}$ the initial value problem
		\begin{subequations}\label{def:cell_prob_wG_sys_reformulated_transformed}
			\begin{align}
			\int_{\samplingdomain{\delta}{x}}\left[\vec{M}\left(x,\frac{y}{\delta}\right) \partial_{t} \grady \wlG(t,x,y) + \vec{R}\left(x,\frac{y}{\delta}\right) \grady \wlG(t,x,y)\right]\cdot \nabla_y v(y) \D y = 0 \eqcomma \label{def:cell_prob_wG_dt_reformulated_transformed}\\
			\int_{\samplingdomain{\delta}{x}} \left[\vec{M}\left(x,\frac{y}{\delta}\right) \grady \wlG(0,x,y) + \vec{R}\left(x,\frac{y}{\delta}\right)\left(\el + \grady \wlM(x,y) \right)\right]\cdot \nabla_y v(y) \D y = 0 \eqdot \label{def:cell_prob_wG_0_reformulated_transformed}
			\end{align}
		\end{subequations}
		The cell corrector $\wlN \in \Ck{\infty}{0,T;\Hper{\samplingdomain{\delta}{x};\R^{\dimMicro}}}$ solves for all $t\in(0,T)$ and all $v\in\Hper{\samplingdomain{\delta}{x};\R^{\dimMicro}}$ the initial value problem
		\begin{subequations}\label{def:cell_prob_wN_sys_reformulated_transformed}
			\begin{align}
			\int_{\samplingdomain{\delta}{x}} \left[\vec{M}\left(x,\frac{y}{\delta}\right) \partial_{t} \grady \wlN(t,x,y) + \vec{R}\left(x,\frac{y}{\delta}\right) \grady \wlN(t,x,y)\right]\cdot \nabla_y v(y) \D y = 0 \eqcomma \label{def:cell_prob_wN_dt_reformulated_transformed}\\
			\int_{\samplingdomain{\delta}{x}} \vec{M}\left(x,\frac{y}{\delta}\right)\left(\el - \grady \wlN(0,x,y) \right)\cdot \nabla_y v(y) \D y = 0 \eqdot \label{def:cell_prob_wN_0_reformulated_transformed}
			\end{align}
		\end{subequations}
	\end{corollary}
	\begin{proof}
		We start with \cref{thm:homogeneous_system}. The reformulation is based on the evaluation of the time derivative of the convolution and the equivalent formulation of integral- as differential equations. Eventually, a change of variables yields the formulation using the sampling domains. For the details, we refer to \cite[Section 4.4.4 and Appendix A.2]{Freese2021}. In \cref{lem:wellposedness_wG} below, we prove the regularity of the cell correctors.
	\end{proof}

	\section{\Wellposedness of microscopic and macroscopic problems}\label{sec:Wellposedness of microscopic and macroscopic problems}
	In this section we drop the dependence on the $x$ variable, since it is only a parameter in the cell problems. Thus, let $x\in\Omega$ be fixed. Furthermore, we abbreviate the solution space of the cell problems by $\spaceVmic\coloneqq \Hper{\samplingdomain{\delta}{x};\R^{\dimMicro}}$, equipped with the inner product
	\begin{align*}
	\bilinearform{}{\phi}{\psi}_{\spaceVmic} = \bilinearform{}{\grady \phi}{\grady \psi}_{\Lp{2}{\samplingdomain{\delta}{x};\R^{\dimMicro}}} \quad \text{for all }\phi,\psi\in\spaceVmic \eqdot
	\end{align*}
	This is a inner product due to the Poincar\'{e} inequality. The induced norm is equivalent to the $\mrm{L}^2$-norm of the gradient, i.e, $\norm{\cdot}{\spaceVmic} = \norm{\grady \cdot}{\Lp{2}{\samplingdomain{\delta}{x}}}$. Additionally, we introduce the weighted inner product $\namemicrobilinearform{m}^{\delta}:\spaceVmic \times \spaceVmic \to \R$ given as
	\begin{align}
	\microkappabilinearform{m}{\phi}{\psi} \coloneqq \int_{\samplingdomain{\delta}{x}}\vec{M}\left(\frac{y}{\delta}\right) \grady \phi(y) \cdot \grady \psi(y) \D y \quad \text{for all }\phi,\psi\in \spaceVmic\eqdot \label{def:bilinearform_m}
	\end{align}
	It is bounded and coercive, since the parameter $\vec{M}^{\delta}$ is assumed to be bounded and positive definite, cf. \cref{def:propertiesM}. Thus, the bilinear form is an inner product on $\spaceVmic$, and we denote its induced norm by $\norm{\cdot}{\namemicrobilinearform{m}^{\delta}}$. Hence, we have an equivalent inner product, which satisfies
	\begin{align}
	\sqrt{\coercivityconstant} \norm{\phi}{\spaceVmic} \leq \norm{\phi}{\namemicrobilinearform{m}^{\delta}} \leq \sqrt{\boundednessconstant} \norm{\phi}{\spaceVmic}\quad \text{for all } \phi\in\spaceVmic \eqdot \label{eq:properties_microbilinearform_m}
	\end{align}
	\begin{lemma}\label{lem:wellposedness_wM_wG0_wN0}
		For every $\secondindex = 1,\dots,\dimMaxwell$ the cell problems \cref{def:cell_prob_wM_reformulated_transformed}, \cref{def:cell_prob_wG_0_reformulated_transformed} and \cref{def:cell_prob_wN_0_reformulated_transformed} are \wellposed\!. Moreover, the solutions are bounded by
		\begin{align*}
		\norm{\wlM}{\namemicrobilinearform{m}^{\delta}} \leq \sqrt{\boundednessconstant 
			\seminorm{\samplingdomainname{\delta}}{}} \eqcomma \quad
		\norm{\wlG(0)}{\namemicrobilinearform{m}^{\delta}} \leq 2 \frac{\genericconstant_{\vec{R}}}{\coercivityconstant} \sqrt{\boundednessconstant 
			\seminorm{\samplingdomainname{\delta}}{}} \eqcomma \quad
		\norm{\wlN(0)}{\namemicrobilinearform{m}^{\delta}} \leq \sqrt{\boundednessconstant 
			\seminorm{\samplingdomainname{\delta}}{}} \eqdot
		\end{align*}
	\end{lemma}
	\begin{proof}
		The three problems \cref{def:cell_prob_wM_reformulated_transformed,def:cell_prob_wG_0_reformulated_transformed,def:cell_prob_wN_0_reformulated_transformed} can be written as: Find $w\in\spaceVmic$ such that 
		\begin{align*}
		\microkappabilinearform{m}{w}{v} = \linearform{b}{v} \quad \text{for all }v\in\spaceVmic \eqdot
		\end{align*}
		with a suitable definition of the linear, bounded functional $b$. Thus, the \wellposedness follows from the Lax--Milgram lemma. For the bound in the $\namemicrobilinearform{m}^{\delta}$-norm on $\wlM$ we choose in \cref{def:cell_prob_wM_reformulated_transformed} $v=\wlM$. Due to \cref{eq:properties_microbilinearform_m} and the Cauchy--Schwarz inequality we find
		\begin{align*}
		\norm{\wlM}{\namemicrobilinearform{m}^{\delta}}^{2} &= \microkappabilinearform{m}{\wlM}{\wlM} = -\int_{\samplingdomain{\delta}{x}} \vec{M}\left(\frac{y}{\delta}\right) \el \cdot \grady \wlM(y) \D y \leq \sqrt{\boundednessconstant \seminorm{\samplingdomainname{\delta}}{}} \norm{\wlM}{\namemicrobilinearform{m}^{\delta}} \eqdot
		\end{align*}
		Similarly, we choose $v=\wlG(0)$ in \cref{def:cell_prob_wG_0_reformulated_transformed} which yields
		\begin{align*}
		&\norm{\wlG(0)}{\namemicrobilinearform{m}^{\delta}}^{2} = \microkappabilinearform{m}{\wlG(0)}{\wlG(0)} = -\int_{\samplingdomain{\delta}{x}} \vec{R}\left(\frac{y}{\delta}\right) \left(\el + \grady \wlM(y) \right) \cdot \grady \wlG(0,y) \D y \\
		&\quad\leq \genericconstant_{\vec{R}} \left(\norm{\el}{\Lp{2}{\samplingdomain{\delta}{x}}} + \norm{ \wlM}{\spaceVmic}\right) \norm{\wlG(0)}{\spaceVmic} \leq 2 \frac{\genericconstant_{\vec{R}}}{\coercivityconstant} \sqrt{\boundednessconstant \seminorm{\samplingdomainname{\delta}}{}} \norm{\wlG(0)}{\namemicrobilinearform{m}^{\delta}} \eqdot
		\end{align*}
		The bound for $\wlN(0)$ follows like that for $\wlM$ by choosing $v=\wlN(0)$ in \cref{def:cell_prob_wN_0_reformulated_transformed}.
	\end{proof}
	The two time-dependent cell problems \cref{def:cell_prob_wG_dt_reformulated_transformed} and \cref{def:cell_prob_wN_dt_reformulated_transformed} are so-called Sobolev equations with respective initial value, which is given in each case in \cref{def:cell_prob_wG_0_reformulated_transformed,def:cell_prob_wN_0_reformulated_transformed}. Results for those problems are found in \cite{Bekkouche2019,Liu2002}. For the following error analysis, we provide two results. Using ideas from \cite{Hipp2017}, we introduce another bilinear form related to the parameter $\vec{R}^{\delta}$.
	\begin{align}
	&\microkappabilinearform{r}{\phi}{\psi} \coloneqq \int_{\samplingdomain{\delta}{x}} \vec{R}\left(\frac{y}{\delta}\right) \grady \phi(y) \cdot \grady \psi(y) \D y  &&\text{for all }\phi,\psi\in \spaceVmic \eqdot \label{def:bilinearform_r}
	\end{align}
	In contrast to the bilinear form $\namemicrobilinearform{m}^{\delta}$ the form $\namemicrobilinearform{r}^{\delta}$ is only bounded but not coercive, since the parameter is only positive semi-definite. We consider a general (variational) Sobolev equation with initial value $w_{0}\in \spaceVmic$
	\begin{subequations}\label{sys:Sobolev_equation}
		\begin{align}
		\microkappabilinearform{m}{\partial_{t} w(t)}{v} + \microkappabilinearform{r}{w(t)}{v} &= 0 \quad \text{for all }v\in\spaceVmic \eqcomma t\in[0,T] \eqdot \label{eq:Sobolev_equation_variational} \\ 
		w(0) &= w_{0} \quad\text{in }\samplingdomain{\delta}{x} \eqdot
		\end{align}
	\end{subequations}
	Due to Riesz-representation theorem we find an operator $\mcal{S}:\spaceVmic\to\spaceVmic$ such that
	\begin{align}
	\microkappabilinearform{r}{\phi}{\psi} = \microkappabilinearform{m}{\mcal{S} \phi}{\psi} \quad \text{for all }\phi,\psi\in\spaceVmic \eqdot \label{def:op_A}
	\end{align}
	We thus get the Sobolev equation as
	\begin{align}
	\microkappabilinearform{m}{\partial_t w(t)}{v} + \microkappabilinearform{m}{\mcal{S}w(t)}{v} = 0 \quad \text{for all }v\in\spaceVmic\eqdot \label{eq:abstract_Cauchy_Sobolev}
	\end{align}
	With respect to the inner product $\namemicrobilinearform{m}^{\delta}$, the operator $\mcal{S}$ inherits by definition its properties from the bilinear form $\namemicrobilinearform{r}^{\delta}$. From \cref{def:propertyR} and \cref{eq:properties_microbilinearform_m} we find that $\mcal{S}$ is bounded
	\begin{align*}
	\seminorm{\microkappabilinearform{m}{\mcal{S}\phi}{\psi}}{} &= \seminorm{\microkappabilinearform{r}{\phi}{\psi}}{} \leq \genericconstant_{\vec{R}} \norm{\phi}{\spaceVmic} \norm{\psi}{\spaceVmic} \leq \frac{\genericconstant_{\vec{R}}}{\coercivityconstant} \norm{\phi}{\namemicrobilinearform{m}^{\delta}} \norm{\psi}{\namemicrobilinearform{m}^{\delta}} \eqdot
	\end{align*}
	Moreover, since $\vec{R}^{\delta}$ is non-negative, the operator $-\mcal{S}$ is dissipative, i.e.,
	\begin{align*}
	\microkappabilinearform{m}{-\mcal{S}\phi}{\phi} = -\microkappabilinearform{r}{\phi}{\phi} \leq 0 \eqdot
	\end{align*} 
	In this setting the operator $\mcal{S}$ also satisfies the range condition with respect to $\spaceVmic$, i.e., $\operatorname{range}(\identityoperator+\mcal{S})=\spaceVmic$. This follows by an application of the Lax--Milgram lemma.
	\begin{lemma}\label{lem:wellposedness_wG}
		For every $\secondindex = 1,\dots,\dimMaxwell$ the cell problems \cref{def:cell_prob_wG_sys_reformulated_transformed} and \cref{def:cell_prob_wN_sys_reformulated_transformed} are well-posed. Moreover, the solutions have the regularity $\wlG$, $\wlN\in\Ck{\infty}{0,T;\spaceVmic}$ and satisfy the stability bounds
		\begin{align*}
		\norm{\wlG(t)}{\namemicrobilinearform{m}^{\delta}} \leq \norm{\wlG(0)}{\namemicrobilinearform{m}^{\delta}} \eqcomma \quad
		\norm{\wlN(t)}{\namemicrobilinearform{m}^{\delta}} \leq \norm{\wlN(0)}{\namemicrobilinearform{m}^{\delta}} \eqdot
		\end{align*}
	\end{lemma}
	\begin{proof}
		Let $\secondindex = 1,\dots,\dimMicro$ be fixed. Note, that \cref{def:cell_prob_wG_dt_reformulated_transformed} is equivalent to the abstract Cauchy problem \cref{eq:abstract_Cauchy_Sobolev} with $w=\wlG$, where the initial value satisfies \cref{def:cell_prob_wG_0_reformulated_transformed}. 
		We just showed that $-\mcal{S}$ is dissipative and satisfies the range condition. By the Lumer--Phillips theorem \cite[Chapter II, Corollary 3.20]{Engel1999} it generates a contraction semigroup $\left(\euler^{-\mcal{S} t}\right)_{t\geq 0}$, i.e., in the $\norm{\cdot}{\namemicrobilinearform{m}^{\delta}}$-norm we get
		\begin{align}\label{eq:contractionSobolev}
		\norm{\euler^{-\mcal{S}t}}{\namemicrobilinearform{m}^{\delta}\gets \namemicrobilinearform{m}^{\delta}} \leq 1 \eqdot
		\end{align}
		Hence, due to \cite[Chapter III, 6.2 Proposition]{Engel1999} the abstract Cauchy problem is \wellposed and the solution is given by
		\begin{align}
		\wlG(t) = \euler^{-\mcal{S}t} \wlG(0) \eqdot \label{eq:Sobolev_equation_solution_semigroup}
		\end{align}
		The initial value $\wlG(0)$ is given as solution of \cref{def:cell_prob_wG_0_reformulated_transformed}, which is \wellposed thanks to \cref{lem:wellposedness_wM_wG0_wN0}. For the regularity of the solution we use the representation of the solution given in \cref{eq:Sobolev_equation_solution_semigroup} and the fact that the operator is bounded. The series representation of the exponential for bounded operators may be differentiated infinitely often, which yields the regularity in time. In turn, the stability bound directly follows from the representation of the solution in \cref{eq:Sobolev_equation_solution_semigroup} and the contraction property \cref{eq:contractionSobolev}. \\
		Exactly the same argumentation is valid for the cell problem \cref{def:cell_prob_wN_sys_reformulated_transformed}.
	\end{proof}
	For the error analysis below, the following $\mrm{H}^{2}$-estimate is essential.
	\begin{theorem}\label{thm:H2_Sobolev}
		Let $\vec{M}\in\Wkp{1}{\infty}{\samplingdomain{\delta}{x}}$ be symmetric and uniformly positive definite, $\vec{R}\in\Wkp{1}{\infty}{\samplingdomain{\delta}{x}}$ be positive semi-definite and assume that the initial value is $\mrm{H}^{2}$-regular. For a solution $w(t,\cdot)\in\Hk{2}{\samplingdomain{\delta}{x}}$, $t\in[0,T]$ of \cref{sys:Sobolev_equation} we get the estimate
		\begin{align*}
		\norm{w(t)}{\Hk{2}{\samplingdomain{\delta}{x}}} \leq \genericconstant 	\norm{w_{0}}{\Hk{2}{\samplingdomain{\delta}{x}}}\eqdot
		\end{align*}
	\end{theorem}
	\begin{proof}
		Let $w_{0}\in\Hk{2}{\samplingdomain{\delta}{x}}$ be given. We rewrite the Sobolev equation \cref{eq:Sobolev_equation_variational} in strong formulation as: Find $w:[0,T]\to \Hk{2}{\samplingdomain{\delta}{x}}$ such that
		\begin{align*}
		\Delta_{\vec{M}} \partial_{t} w(t) + \Delta_{\vec{R}} w(t) = 0\quad \text{in }\samplingdomain{\delta}{x} \eqdot
		\end{align*}
		The operators $\Delta_{\vec{M}} = \div\left(\vec{M}(y) \nabla \cdot \right)$ and $\Delta_{\vec{R}} = \div\left(\vec{R}(y) \nabla \cdot \right)$ are weighted Laplace operators. Note, that $\Delta_{\vec{M}}:\Hk{2}{\samplingdomain{\delta}{x}} \to \Lp{2}{\samplingdomain{\delta}{x}}$ is invertible due to the properties of $\vec{M}$, cf. \cref{def:propertiesM}. 
		The operator $\mcal{S}$ defined in \cref{def:op_A} as operator on $\Hk{2}{\samplingdomain{\delta}{x}}$ may be written as $\mcal{S} = \Delta_{\vec{M}}^{-1} \Delta_{\vec{R}}$. Next, we consider the closely related operator $\Delta_{\vec{R}} \Delta_{\vec{M}}^{-1}:\Lp{2}{\samplingdomain{\delta}{x}}\to\Lp{2}{\samplingdomain{\delta}{x}}$. We show that this operator is also monotone and bounded with respect to $\Lp{2}{\samplingdomain{\delta}{x}}$ equipped with the weighted inner product
		\begin{align*}
		\bilinearform{}{\Phi}{\Psi}_{\Delta_{\vec{M}}^{-1}} = \bilinearform{}{\Phi}{-\Delta_{\vec{M}}^{-1}\Psi} \quad \text{for all }\Phi,\Psi\in\Lp{2}{\samplingdomain{\delta}{x}} \eqdot
		\end{align*}
		Note, that the weighted Laplacian has a negative spectrum, and thus the negative operator induces an inner product. Integration by parts yields
		\begin{align*}
		\bilinearform{}{\Delta_{\vec{R}}\Delta_{\vec{M}}^{-1} \Phi}{\Phi}_{\Delta_{\vec{M}}^{-1}} = \bilinearform{}{\Delta_{\vec{R}}\Delta_{\vec{M}}^{-1}\Phi}{-\Delta_{\vec{M}}^{-1}\Phi} = \bilinearform{}{\vec{R}\nabla\left(\Delta_{\vec{M}}^{-1}\Phi\right)}{\nabla \left(\Delta_{\vec{M}}^{-1}\Phi\right)} \geq 0 \eqdot
		\end{align*}
		Additionally, using again integration by parts, the boundedness of $\vec{R}$ as well as the positive definiteness of $\vec{M}$ and the Cauchy--Schwarz inequality we get
		\begin{align*}
		&\seminorm{\bilinearform{}{\Delta_{\vec{R}}\Delta_{\vec{M}}^{-1}\Phi}{\Psi}_{\Delta_{\vec{M}}^{-1}}}{} = \seminorm{\bilinearform{}{\vec{R}\nabla\left(\Delta_{\vec{M}}^{-1} \Phi\right)}{\nabla\left(\Delta_{\vec{M}}^{-1} \Psi\right)}}{} \\
		&\quad\leq \frac{\genericconstant_{\vec{R}}}{\coercivityconstant} \seminorm{\bilinearform{}{\vec{M} \nabla\left(\Delta_{\vec{M}}^{-1} \Phi\right)}{\nabla\left(\Delta_{\vec{M}}^{-1} \Psi\right)}}{} = \frac{\genericconstant_{\vec{R}}}{\coercivityconstant} \seminorm{- \bilinearform{}{\div\left(\vec{M} \nabla\left(\Delta_{\vec{M}}^{-1} \Phi\right) \right)}{\Delta_{\vec{M}}^{-1} \Psi}}{} \\
		&\quad= \frac{\genericconstant_{\vec{R}}}{\coercivityconstant} \seminorm{\bilinearform{}{\Phi}{-\Delta_{\vec{M}}^{-1} \Psi}}{} = \frac{\genericconstant_{\vec{R}}}{\coercivityconstant} \seminorm{\bilinearform{}{\Phi}{\Psi}_{\Delta_{\vec{M}}^{-1}}}{} \leq \frac{\genericconstant_{\vec{R}}}{\coercivityconstant} \norm{\Phi}{\Delta_{\vec{M}}^{-1}} \norm{\Psi}{\Delta_{\vec{M}}^{-1}} \eqdot
		\end{align*}
		Thus, the operator $-\Delta_{\vec{R}} \Delta_{\vec{M}}^{-1}$ is dissipative and generates a contraction semi-group $\euler^{-\Delta_{\vec{R}}\Delta_{\vec{M}}^{-1} t}$.
		Both $-\Delta_{\vec{M}}^{-1} \Delta_{\vec{R}}$ and $-\Delta_{\vec{R}} \Delta_{\vec{M}}^{-1}$ are bounded operators. We use the series representation for their generated semi-groups and define for $n\in\N$ the partial sum
		\begin{align*}
		s_{n}(t) \coloneqq \sum\nolimits_{k=0}^{n}\left(- \Delta_{\vec{R}} \Delta_{\vec{M}}^{-1}\right)^{k} \frac{t^{k}}{k!} \Delta_{\vec{M}} w(0) \eqdot
		\end{align*}
		Since $\Delta_{\vec{M}} w(0)\in\Lp{2}{\samplingdomain{\delta}{x}}$ and due to the properties of $-\Delta_{\vec{R}}\Delta_{\vec{M}}^{-1}$ we just showed, we find the convergence in $\Lp{2}{\samplingdomain{\delta}{x}}$ as
		\begin{align*}
		s_{n}(t) \to s(t) = \euler^{-\Delta_{\vec{R}} \Delta_{\vec{M}}^{-1} t} \Delta_{\vec{M}} w(0) \quad\text{as } n\to \infty\eqdot
		\end{align*}
		Moreover, for $z_{n}(t) \coloneqq \sum_{k=0}^{n}\left(-\Delta_{\vec{M}}^{-1} \Delta_{\vec{R}}\right)^{k} \frac{t^{k}}{k!} w(0)$ we find $\Delta_{\vec{M}} z_{n}(t) = s_{n}(t)$, and
		\begin{align*}
		z_{n}(t) \to w(t) = \euler^{-\Delta_{\vec{M}}^{-1} \Delta_{\vec{R}} t} w(0)\quad \text{as }n\to\infty \eqdot
		\end{align*} 
		This enables us to use the fact that $\Delta_{\vec{M}}$ is closed. To be precise, we conclude from
		\begin{align*}
		z_{n}(t) \to w(t)& \quad \text{ as }n\to\infty \eqcomma \quad
		\Delta_{\vec{M}} z_{n}(t) = s_{n}(t) \to s(t) \quad \text{ as }n\to\infty \eqcomma
		\end{align*}
		that $w(t)\in\operatordomain\left(\Delta_{\vec{M}}\right)$ and
		\begin{align*}
		\Delta_{\vec{M}} \left(\euler^{-\Delta_{\vec{M}}^{-1} \Delta_{\vec{R}} t} w(0)\right) = \Delta_{\vec{M}} w(t) = s(t) = \euler^{-\Delta_{\vec{R}} \Delta_{\vec{M}}^{-1} t} \Delta_{\vec{M}} w(0)\eqdot
		\end{align*}
		Eventually, we find
		\begin{align*}
		\norm{w(t)}{\Hk{2}{\samplingdomain{\delta}{x}}} &\leq \genericconstant \norm{\Delta_{\vec{M}} w(t)}{\Lp{2}{\samplingdomain{\delta}{x}}} 
		= \genericconstant \norm{ s(t)}{\Lp{2}{\samplingdomain{\delta}{x}}} 
		\\
		&\leq \genericconstant \norm{ \euler^{-\Delta_{\vec{R}} \Delta_{\vec{M}}^{-1} t}}{\Delta_{\vec{M}}^{-1} \gets \Delta_{\vec{M}}^{-1}} \norm{\Delta_{\vec{M}} w(0)}{\Delta_{\vec{M}}^{-1}} \leq \genericconstant \norm{w(0)}{\Hk{2}{\samplingdomain{\delta}{x}}} \eqdot
		\end{align*}
	\end{proof}

	\subsection{\Wellposedness of the integro-differential homogeneous system}\label{sec:Well-posedness of the integro-differential homogeneous system}
	Let us give some properties of the effective parameters.
	\begin{lemma}\label{lem:bounded_effective_parameters}
		The effective parameter $\vec{M}^{\eff}$ is positive definite and bounded with the same constant $\coercivityconstant$ as $\vec{M}^{\delta}$, and $\vec{R}^{\eff}$ is positive semi-definite and bounded. The time-dependent parameters satisfy $\vec{G}^{\eff}$, $\vec{J}^{\eff} \in \Ck{\infty}{0,T;\Lp{\infty}{\computationaldomain;\R^{\dimMaxwell \times \dimMaxwell}}}$.
	\end{lemma}
	\begin{proof}
		The boundedness of $\vec{M}^{\eff}$ is shown in \cite[Chapter 1.4]{Jikov1994}, while the coercivity with the same constant $\alpha$ as in \cref{def:propertiesM} follows as in \cite[Chapter 2.3]{Bensoussan1978}. The positive semi-definiteness of $\vec{R}^{\delta}$ directly implies the one of $\vec{R}^{\eff}$ by its definition in \cref{def:R_eff_reformulated_transformed}. For fixed $x\in\computationaldomain$ the boundedness follows using the definition in \cref{def:R_eff_reformulated_transformed} by
		\begin{align*}
		\seminorm{\vec{R}^{\eff}(x)_{\firstindex,\secondindex}}{} 
		&\leq \frac{\genericconstant_{\vec{R}}}{\seminorm{\samplingdomainname{\delta}}{}} \left(\norm{\el}{\Lp{2}{\samplingdomain{\delta}{x}}}\left(\norm{\ek}{\Lp{2}{\samplingdomain{\delta}{x}}} + \norm{\wkM(x,\cdot)}{\spaceVmic}\right) \right. \\
		&\quad\left.+ \norm{\wlM(x,\cdot)}{\spaceVmic} \left(\norm{\ek}{\Lp{2}{\samplingdomain{\delta}{x}}} + \norm{ \wkM(x,\cdot)}{\spaceVmic}\right)\right) \leq 4 \frac{\genericconstant_{\vec{R}} \boundednessconstant}{\coercivityconstant} \eqcomma
		\end{align*}
		where we used \cref{lem:wellposedness_wM_wG0_wN0} and the coercivity of $\vec{M}^{\delta}$.	Next, consider the convolution kernel. Using \cref{def:G_eff_reformulated_transformed}, \cref{lem:wellposedness_wM_wG0_wN0,lem:wellposedness_wG} together with \cref{eq:contractionSobolev} yields
		\begin{align*}
		\seminorm{\vec{G}^{\eff}(t,x)_{\firstindex,\secondindex}}{} 
		&\leq \frac{\genericconstant_{\vec{R}}}{\seminorm{\samplingdomainname{\delta}}{}} \norm{ \wlG(t,x,\cdot)}{\spaceVmic} \left(\norm{\ek}{\Lp{2}{\samplingdomain{\delta}{x}}} + \norm{ \wlM(x,\cdot)}{\spaceVmic}\right) \\
		&\leq 2 \frac{\genericconstant_{\vec{R}}\sqrt{\boundednessconstant \seminorm{\samplingdomainname{\delta}}{}}}{\coercivityconstant \seminorm{\samplingdomainname{\delta}}{}} \norm{\wlG(0,x,\cdot)}{\namemicrobilinearform{m}^{\delta}} \leq 4 \left(\frac{\genericconstant_{\vec{R}}}{\coercivityconstant}\right)^{2} \boundednessconstant \eqdot
		\end{align*}
		Finally for the extra source we get from \cref{def:J_eff_reformulated_transformed} with the same techniques as above
		\begin{align*}
		\seminorm{\vec{J}^{\eff}(t,x)_{\firstindex,\secondindex}}{} 
		\leq 2 \frac{\genericconstant_{\vec{R}}\boundednessconstant}{\coercivityconstant} \eqdot
		\end{align*}	
		Therefore, we find $\vec{G}^{\eff},$ $\vec{J}^{\eff} \in \Lp{\infty}{0,T;\Lp{\infty}{\computationaldomain;\R^{\dimMaxwell \times \dimMaxwell}}}$. As pointed out in \cref{lem:wellposedness_wG}, the cell correctors $\wlG$, $\wlN$ are $\mrm{C}^{\infty}$ in time. Thus, a direct consequence is the smoothness in time, i.e., $\vec{G}^{\eff},$ $\vec{J}^{\eff} \in \Ck{\infty}{0,T;\Lp{\infty}{\Omega;\R^{\dimMaxwell\times\dimMaxwell}}}$.
	\end{proof}
	For the variational formulation of the effective Maxwell system \cref{sys:reformulated_effective_Maxwell_system} we introduce bilinear forms $m^{\eff},$ $r^{\eff},$ $a:\spaceVmac \times \spaceVmac \to \R$ such that for every $\vec{\Phi},$ $\vec{\Psi} \in \spaceVmac$
	\begin{subequations}
		\begin{align}
		\bilinearform{m^{\eff}}{\vec{\Phi}}{\vec{\Psi}} &\coloneqq \bilinearform{}{\vec{M}^{\eff}\vec{\Phi}}{\vec{\Psi}} \eqcomma \label{def:bilinearform_m_eff}\\
		\bilinearform{r^{\eff}}{\vec{\Phi}}{\vec{\Psi}} &\coloneqq \bilinearform{}{\vec{R}^{\eff}\vec{\Phi}}{\vec{\Psi}} \eqcomma \label{def:bilinearform_r_eff}\\
		\bilinearform{a}{\vec{\Phi}}{\vec{\Psi}} &\coloneqq \bilinearform{}{\vec{A}\vec{\Phi}}{\vec{\Psi}}\eqdot \label{def:bilinearform_a}
		\end{align}
		Moreover, for $t\in[0,T]$ define $g^{\eff}:[0,T]\times \spaceVmac \times \spaceVmac \to\R$ such that
		\begin{align}
		\bilinearform{g^{\eff}}{t;\vec{\Phi}}{\vec{\Psi}} &\coloneqq \bilinearform{}{\vec{G}^{\eff}(t)\vec{\Phi}}{\vec{\Psi}}\quad \text{for all }\vec{\Phi}, \vec{\Psi}\in \spaceVmac\eqdot \label{def:bilinearform_g_eff}
		\end{align}
	\end{subequations}
	Thanks to the bounds on the parameters, we immediately get the continuity of all bilinear forms, i.e., for all $\vec{\Phi},$ $\vec{\Psi}\in\Lp{2}{\Omega;\R^{\dimMaxwell}}$ and $t\geq 0$ we find
	\begin{align}
	\seminorm{\bilinearform{m^{\eff}}{\vec{\Phi}}{\vec{\Psi}}}{} &\leq \boundednessconstant \norm{\vec{\Phi}}{} \norm{\vec{\Psi}}{} \eqcomma \nonumber \\
	\seminorm{\bilinearform{r^{\eff}}{\vec{\Phi}}{\vec{\Psi}}}{} &\leq 4 \frac{\genericconstant_{\vec{R}}\boundednessconstant}{\coercivityconstant} \norm{\vec{\Phi}}{} \norm{\vec{\Psi}}{} \eqcomma \label{eq:boundedness_bilinearform_r_eff} \\
	\seminorm{\bilinearform{g^{\eff}}{t;\vec{\Phi}}{\vec{\Psi}}}{} &\leq 4 \left(\frac{\genericconstant_{\vec{R}}}{\coercivityconstant}\right)^{2} \boundednessconstant \norm{\vec{\Phi}}{} \norm{\vec{\Psi}}{}\eqdot \label{eq:boundedness_bilinearform_g_eff}
	\end{align}
	Moreover, $\bilinearform{m^{\eff}}{\cdot}{\cdot}$ and $\bilinearform{r^{\eff}}{\cdot}{\cdot}$ are bounded from below by
	\begin{align}
	\bilinearform{m^{\eff}}{\vec{\Phi}}{\vec{\Phi}} \geq \coercivityconstant \norm{\vec{\Phi}}{}^{2} \eqcomma \quad
	\bilinearform{r^{\eff}}{\vec{\Phi}}{\vec{\Phi}} \geq 0 \eqdot \label{eq:definiteness_m_eff_r_eff}
	\end{align}
	With these definitions we search for $\effsolutionu(t) \in \spaceVmac$ such that
	\begin{align}
	\begin{split}\label{eq:bilinear_effective_Maxwell}
	\bilinearform{m^{\eff}}{\partial_{t}\effsolutionu(t)}{\vec{\Phi}} &+ \bilinearform{r^{\eff}}{\effsolutionu(t)}{\vec{\Phi}} + \int_{0}^{t}\bilinearform{g^{\eff}}{t-s;\effsolutionu(s)}{\vec{\Phi}} \D s + \bilinearform{a}{\effsolutionu(t)}{\vec{\Phi}}\\
	&= \bilinearform{m^{\eff}}{\vec{f}(t)}{\vec{\Phi}} - \bilinearform{}{\vec{J}^{\eff}(t)\vec{u}_{0}}{\vec{\Phi}} \quad \text{for all }\vec{\Phi}\in \spaceVmac \eqcomma
	\end{split}\\
	\effsolutionu(0) &= \vec{u}_{0} \qquad \text{in }\computationaldomain \eqcomma \nonumber
	\end{align}
	where $\vec{f}\in\Lp{2}{0,T;\Lp{2}{\computationaldomain}}$ is such that
	\begin{align*}
	\bilinearform{m^{\eff}}{\vec{f}(t)}{\vec{\Phi}} = \bilinearform{}{\vec{g}(t)}{\vec{\Phi}} \quad \text{for all }\vec{\Phi}\in \spaceVmac \eqdot
	\end{align*}

	\begin{theorem}\label{thm:wellposedness_Maxwell_integral}
		Let the assumptions of \cref{thm:homogeneous_system} be satisfied and assume that $\vec{R}^{\delta}$ is positive semi-definite. Then, problem \cref{eq:bilinear_effective_Maxwell} has a unique solution satisfying
		\begin{align}
		\norm{\effsolutionu(t)}{} \leq \euler^{\genericconstant_{\vec{G}^{\eff}}(t)} \left[\frac{1}{\alpha} t \norm{\vec{g}}{\Lp{\infty}{0,t;\Lp{2}{\computationaldomain;\R^{\dimMaxwell}}}} + \left(1+ \frac{1}{\alpha} \norm{\vec{J}^{\eff}}{\Lp{1}{0,t;\Lp{\infty}{\computationaldomain;\R^{\dimMaxwell\times\dimMaxwell}}}}\right) \norm{\solutionu_{0}}{} \right] \eqcomma
		\label{eq:stability_Maxwell_integral}
		\end{align}
		where the exponential growth is determined by
		\begin{align*}
		\genericconstant_{\vec{G}^{\eff}}(t) = \frac{1}{\coercivityconstant} \intt{\norm{\vec{G}^{\eff}}{\Lp{1}{0,s;\Lp{\infty}{\computationaldomain;\R^{\dimMaxwell\times\dimMaxwell}}}}}{s} \eqdot
		\end{align*}
	\end{theorem}
	\begin{proof}
		The proof of the \wellposedness relies on the Faedo--Galerkin method. The details can be found in \cite[Proposition 1]{Bossavit2005}, where also an estimate similar to \cref{eq:stability_Maxwell_integral} is proved. Our bound results from these ideas but provides precise constants, in particular the growth due to the Gronwall estimate. Consider the system \cref{eq:bilinear_effective_Maxwell} with test function $\vec{\Phi} = \effsolutionu(t)$. The product rule, as well as \cref{eq:definiteness_m_eff_r_eff} and the skew-adjointness of the Maxwell operator yields
		\begin{align*}
		&\alpha \left(\frac{d}{dt} \norm{\effsolutionu(t)}{}\right) \norm{\effsolutionu(t)}{} \leq \frac{1}{2} \frac{d}{dt} \bilinearform{m^{\eff}}{\effsolutionu(t)}{\effsolutionu(t)} = \bilinearform{m^{\eff}}{\partial_{t} \effsolutionu(t)}{\effsolutionu(t)} \\
		&\quad\leq \seminorm{\bilinearform{}{\vec{g}(t)}{\effsolutionu(t)}}{} + \seminorm{\bilinearform{}{\vec{J}^{\eff}(t) \solutionu_{0}}{\effsolutionu(t)}}{} + \seminorm{\intt{\bilinearform{g^{\eff}}{t-s;\effsolutionu(s)}{\effsolutionu(t)}}{s}}{} \eqdot
		\end{align*}
		Applying the Cauchy-Schwarz inequality in the last line, dividing by $\norm{\effsolutionu(t)}{}$ and using the H\"older inequality yields
		\begin{align*}
		\alpha \frac{d}{dt} \norm{\effsolutionu(t)}{} &\leq \norm{\vec{g}(t)}{} + \norm{\vec{J}^{\eff}(t) }{\Lp{\infty}{\computationaldomain; \R^{\dimMaxwell \times\dimMaxwell}}} \norm{ \solutionu_{0}}{ } \\
		&\quad+ \intt{\norm{\vec{G}^{\eff}(t-s)}{\Lp{\infty}{\computationaldomain; \R^{\dimMaxwell\times\dimMaxwell}}}}{s}\norm{\effsolutionu}{\Lp{\infty}{0,t; \Lp{2}{\computationaldomain; \R^{\dimMaxwell}}}} \eqdot
		\end{align*}
		Next, we integrate over $[0,t]$, resulting in
		\begin{align*}
		\alpha \norm{\effsolutionu(t)}{} &\leq \alpha \norm{\effsolutionu(0)}{} + \intt{\norm{\vec{g}(s)}{}}{s} + \intt{\norm{\vec{J}^{\eff}(s) }{\Lp{\infty}{\computationaldomain; \R^{\dimMaxwell \times\dimMaxwell}}}}{s} \norm{ \solutionu_{0}}{ } \\
		&\quad+ \intt{\int_{0}^{s} \norm{\vec{G}^{\eff}(s-r)}{\Lp{\infty}{\computationaldomain; \R^{\dimMaxwell\times\dimMaxwell}}} \D r\norm{\effsolutionu}{\Lp{\infty}{0,s; \Lp{2}{\computationaldomain; \R^{\dimMaxwell}}}}}{s} \eqdot
		\end{align*}
		The final step is to take the supremum over $[0,t]$ to obtain
		\begin{align*}
		\norm{\effsolutionu}{\Lp{\infty}{0,t;\Lp{2}{\computationaldomain;\R^{\dimMaxwell}}}} \leq 
		C_0(t) + \intt{C_1(s) \norm{\effsolutionu}{\Lp{\infty}{0,s;\Lp{2}{\computationaldomain;\R^{\dimMaxwell}}}}}{s} \eqcomma	
		\end{align*}
		where
		\begin{align*}
		\begin{split}
		C_{0}(t) &\coloneqq \left(1 + \frac{1}{\alpha} \intt{\norm{\vec{J}^{\eff}(s) }{\Lp{\infty}{\computationaldomain; \R^{\dimMaxwell \times\dimMaxwell}}}}{s}\right) \norm{ \solutionu_{0}}{ } + \frac{1}{\alpha} \intt{\norm{\vec{g}(s)}{}}{s}
		\end{split}\eqcomma \\
		C_{1}(t) &\coloneqq \frac{1}{\alpha} \intt{\norm{\vec{G}^{\eff}(t-s)}{\Lp{\infty}{\computationaldomain; \R^{\dimMaxwell\times\dimMaxwell}}}}{s} = \frac{1}{\alpha} \norm{\vec{G}^{\eff}}{\Lp{1}{0,t;\Lp{\infty}{\computationaldomain;\R^{\dimMaxwell\times\dimMaxwell}}}} \eqdot 
		\end{align*}
		We apply Gronwall's inequality, which yields
		\begin{align*}
		\norm{\effsolutionu}{\Lp{\infty}{0,t;\Lp{2}{\computationaldomain;\R^{\dimMaxwell}}}} &\leq \euler^{\intt{C_{1}(s)}{s}} C_0(t) =
		\exp\left({\genericconstant_{\vec{G}^{\eff}}(t)}\right) C_0(t) \eqdot
		\end{align*}
		Since the supremum is bounded, we get the same estimate for all $s\in[0,t]$, i.e., \cref{eq:stability_Maxwell_integral}.
	\end{proof}

	\section{The Finite Element Heterogeneous Multiscale Method}\label{sec:The Finite Element Heterogeneous Multiscale Method}
	For the space discretization we choose the finite element method. The implementation we strive for uses hexahedral elements. Hence, from now on, assume that we have an adjacent and shape regular family of triangulations $\{\mcal{T}_{H}\}_{H>0}$ of the domain $\computationaldomain$ in parallelepipeds. Other choices of elements are possible and the analysis transfers directly to those variations. The parameter $ H $ represents the maximum diameter of all elements $K\in\mcal{T}_{H}$. All quantities that are labeled with an $H$ are macroscopic expressions, as the triangulation $\mcal{T}_{H}$. For microscopic quantities we use $h$ instead. We define the space of polynomials with maximal degree $\ell$ in the first, $m$ in the second and $n$ in the third component on $K\in\mcal{T}_{H}$ as $\mcal{Q}^{\ell,m,n}(K)$.
	On every element $K\in\mcal{T}_{H}$ of the triangulation we choose a quadrature formula consisting of $Q_{K}\in\N$ quadrature points $x_{K}^{\quadratureindex}\in\overline{K}$ and weights $\weight_{K}^{\quadratureindex}\in\R_{\geq 0}$, $\quadratureindex=1,\dots,Q_{K}$. We assume that the quadrature is exact for polynomials in $\mcal{Q}^{2\ell,2\ell,2\ell}(K)$, $\ell\in\N$. Hence, for $p\in\mcal{Q}^{2\ell,2\ell,2\ell}(K)$ we have
	\begin{align}
	\int_{K} p(x) \D x = \sum\nolimits_{\quadratureindex = 1}^{Q_{K}} \weight_{K}^{\quadratureindex} p(x_{K}^{\quadratureindex}) \eqdot \label{ass:quadrature}
	\end{align}
	Note that we use positive weights, e.g. Gaussian rules, to ensure that the resulting bilinear forms keep, for instance, their positivity.
	
	\subsection{N\'{e}d\'{e}lec finite elements}\label{sec:Nedelec Finite Elements}
	For the macroscopic Maxwell system we use $\Hcurl{\computationaldomain}$-conforming N\'{e}d\'{e}lec elements. For $\ell\in\N$ we define
	\begin{align*}
	\mcal{Q}^\ell_\text{N\'{e}délec}(K) \coloneqq \mcal{Q}^{\ell-1,\ell,\ell}(K)\times \mcal{Q}^{\ell,\ell-1,\ell}(K) \times \mcal{Q}^{\ell,\ell,\ell-1}(K) \quad \text{ for all }K\in\mcal{T}_H \eqdot
	\end{align*}
	Denote the space of N\'{e}d\'{e}lec's elements of the first type of order $\ell\in\N$ by
	\begin{align*}
	\Nedelec{\ell}{\mcal{T}_{H}}\coloneqq \left\lbrace v_H\in \Hcurl{\Omega}\, :\, v_H|_K\in \mcal{Q}_\text{N\'{e}d\'{e}lec}^\ell(K)\text{ for all }K\in\mcal{T}_H\right\rbrace \eqcomma
	\end{align*}
	and further define
	\begin{align*}
	\DirichletNedelec{\ell}{\mcal{T}_{H}}\coloneqq \Hcurldirichlet{\Omega} \cap \Nedelec{\ell}{\mcal{T}_{H}} \eqdot
	\end{align*}
	We obtain the following interpolation error estimate that can be found in \cite[Theorem 6.6]{Monk2003}.
	\begin{theorem}\label{thm:Nedelec_interpolation}
		Let $u\in\Hk{\ell+1}{\Omega;\R^{3}}$. There exists a global interpolation operator\\ $\mcal{I}_H:\Hk{\ell+1}{\Omega;\R^{3}} \rightarrow \Nedelec{\ell}{\mcal{T}_{H}}$ for N\'{e}d\'{e}lec elements of the first type, such that
		\begin{align}
		\lVert u - \mcal{I}_H u\rVert_{\Hcurl{\Omega}}\leq C H^\ell\lvert u \rvert_{\Hk{\ell+1}{\Omega;\R^{3}}}\eqdot \label{def:Nedelec_interpolation}
		\end{align}
	\end{theorem}
	
	\subsection{Lagrange finite elements}\label{sec:Lagrange Finite Elements}
	The cell problems are discretized using standard Lagrangian finite elements. Therefore, as above, we introduce a triangulation $\mcal{T}_{h}$ of the sampling domain $\samplingdomain{\delta}{x}$. The space of Lagrangian elements with periodic boundary conditions is given as
	\begin{align*}
	\spaceVh \coloneqq \left\lbrace v_{h}\in\Hk{1}{\samplingdomain{\delta}{x}}\with v_{h}|_{K} \in \mcal{Q}^{k,k,k}(K)\text{ for all }K\in\mcal{T}_{h}\right\rbrace\eqdot
	\end{align*}
	We get an interpolation error estimate as in \cite[Theorem 3.2.1]{Ciarlet2002} or \cite[Theorem 6.11]{Monk2003}.
	\begin{theorem}\label{thm:Lagrange_interpolation}
		Let $u\in\Hk{k+1}{\samplingdomain{\delta}{x};\R}$. For the triangulation $\mcal{T}_{h}$ of $\samplingdomain{\delta}{x}$ there exists an interpolation operator $\Pi_{h}:\Hk{k+1}{\samplingdomain{\delta}{x};\R}\to \spaceVh$ such that the following estimate holds
		\begin{align}
		\norm{u - \Pi_{h}u}{\Hk{1}{\samplingdomain{\delta}{x};\R}}\leq C h^{k} \seminorm{u}{\Hk{k+1}{\samplingdomain{\delta}{x};\R}}\eqdot \label{def:Lagrange_interpolation}
		\end{align}
	\end{theorem}

	\subsection{The HMM framework and the application to Maxwell's equations}\label{sec:The HMM framework and the application to Maxwell's equations}
	Consider the effective system given in \cref{eq:bilinear_effective_Maxwell}. We choose the finite dimensional subspace of $\spaceVmac$, cf. \cref{def:spaceVmac}, as
	\begin{align}\label{def:spaceVH}
	\spaceVH \coloneqq \DirichletNedelec{\ell}{\mcal{T}_{H}} \times \Nedelec{\ell}{\mcal{T}_{H}}^{N_{E}}\times \Nedelec{\ell}{\mcal{T}_{H}} \subseteq \spaceVmac \eqdot
	\end{align}
	On $\spaceVH$ we use the inner product
	\begin{align*}
	\bilinearform{}{\vec{\Phi}_{H}}{\vec{\Psi}_{H}}_{H} = \sum\nolimits_{K\in\mcal{T}_{H}}\sum\nolimits_{\quadratureindex=1}^{Q_{K}} \weight_{K}^{\quadratureindex} \vec{\Phi}_{H}(x_{K}^{\quadratureindex}) \cdot \vec{\Psi}_{H}(x_{K}^{\quadratureindex})\quad \text{for all }\vec{\Phi}_{H},\vec{\Psi}_{H}\in\spaceVH \eqdot
	\end{align*}
	The discrete counterparts to the bilinear forms defined in \cref{def:bilinearform_m_eff} - \cref{def:bilinearform_g_eff} are $m_{H}^{\eff},$ $r_{H}^{\eff},$ $a_{H}:\spaceVH \times\spaceVH \to \R$ such that for every $\vec{\Phi}_{H},$ $\vec{\Psi}_{H}\in\spaceVH$
	\begin{align}
	\bilinearform{m_{H}^{\eff}}{\vec{\Phi}_{H}}{\vec{\Psi}_{H}} &\coloneqq \bilinearform{}{\vec{M}^{\eff}\vec{\Phi}_{H}}{\vec{\Psi}_{H}}_{H} \eqcomma \quad
	\bilinearform{r_{H}^{\eff}}{\vec{\Phi}_{H}}{\vec{\Psi}_{H}} \coloneqq \bilinearform{}{\vec{R}^{\eff}\vec{\Phi}_{H}}{\vec{\Psi}_{H}}_{H} \eqcomma \nonumber
	\\
	\bilinearform{a_{H}}{\vec{\Phi}_{H}}{\vec{\Psi}_{H}} &\coloneqq \bilinearform{}{\vec{A}_{H}\vec{\Phi}_{H}}{\vec{\Psi}_{H}}_{H} \eqcomma \label{def:bilinearform_a_H}
	\end{align}
	where $\vec{A}_{H}$ is the discretization of $\vec{A}$. Moreover, for $t\in[0,T]$ define $g_{H}^{\eff}:[0,T]\times\spaceVH \times \spaceVH \to\R$ such that
	\begin{align*}
	\bilinearform{g_{H}^{\eff}}{t;\vec{\Phi}_{H}}{\vec{\Psi}_{H}} &\coloneqq \bilinearform{}{\vec{G}^{\eff}(t)\vec{\Phi}_{H}}{\vec{\Psi}_{H}}_{H} \quad \text{for all }\vec{\Phi}_{H}, \vec{\Psi}_{H} \in\spaceVH \eqdot
	\end{align*}
	The semi-discrete formulation of the effective Maxwell system is to find $\effsolutionu_{H}:[0,T]\to \spaceVH$ such that
	\begin{align}
	\begin{split}\label{eq:bilinear_discrete_effective_Maxwell}
	&\bilinearform{m^{\eff}_{H}}{\partial_{t}\effsolutionu_{H}(t)}{\vec{\Phi}_{H}} + \bilinearform{r^{\eff}_{H}}{\effsolutionu_{H}(t)}{\vec{\Phi}_{H}} + \int_{0}^{t}\bilinearform{g^{\eff}_{H}}{t-s;\effsolutionu_{H}(s)}{\vec{\Phi}_{H}} \D s \\
	&+ \bilinearform{a_{H}}{\effsolutionu_{H}(t)}{\vec{\Phi}_{H}} 
	= \bilinearform{m^{\eff}_{H}}{\vec{f}_{H}(t)}{\vec{\Phi}_{H}} - \bilinearform{}{\vec{J}^{\eff}(t)\solutionu_{0,H}}{\vec{\Phi}_{H}}_{H} \quad \text{for all }\vec{\Phi}_{H}\in\spaceVH \eqcomma
	\end{split} \\
	&\effsolutionu_{H}(0) = \solutionu_{0,H} \eqcomma \nonumber
	\end{align}
	where $\vec{f}_{H}$ is defined such that for an approximation $\vec{g}_{H}$ of $\vec{g}$ we find
	\begin{align*}
	\bilinearform{m_{H}^{\eff}}{\vec{f}_{H}(t)}{\vec{\Phi}_{H}} = \bilinearform{}{\vec{g}_{H}(t)}{\vec{\Phi}_{H}}_{H} \quad \text{for all }\vec{\Phi}_{H} \in \spaceVH \eqcomma
	\end{align*}
	and where $\solutionu_{0,H} \in\spaceVH$ is an approximation of the initial value $\solutionu_{0}\in\spaceVmac$.
	
	The system \cref{eq:bilinear_discrete_effective_Maxwell} is semi-discrete since it involves the analytic effective parameters. To be precise, the information from the microscopic scale is missing. This situation is exactly where the Heterogeneous Multiscale Method (HMM), introduced in \cite{Engquist2003} and exposed in \cite{AbdulleEEngquist2012}, is applicable. 
	\begin{remark}
		Note that the subspace property $\spaceVH\subseteq\spaceVmac$ is one ingredient of a \textit{conforming} finite element method. We point out that this is not restrictive and that it is possible to use a finite dimensional space that is not a subspace of $\spaceVmac$, e.g, a discontinuous Galerkin approach. See also \cite[Remark 4.6]{Hochbruck2018}.
	\end{remark}
	
	\subsection{Approximating the effective parameters on the microscale}\label{sec:Approximating the effective parameters on the microscale}
	With the definition of the discrete bilinear form we get for $\vec{\Phi}_{H}$, $\vec{\Psi}_{H}\in\spaceVH$
	\begin{align*}
	\begin{split}
	\bilinearform{}{\vec{M}^{\eff}\vec{\Phi}_{H}}{\vec{\Psi}_{H}} \approx \bilinearform{m_{H}^{\eff}}{\vec{\Phi}_{H}}{\vec{\Psi}_{H}}
	=\sum\nolimits_{K\in\mcal{T}_{H}} \sum\nolimits_{\quadratureindex = 1}^{Q_{K}} \weight_{K}^{\quadratureindex} \vec{M}^{\eff}(x_{K}^{\quadratureindex}) \vec{\Phi}_{H}(x_{K}^{\quadratureindex}) \cdot \vec{\Psi}_{H}(x_{K}^{\quadratureindex}) \eqdot
	\end{split}
	\end{align*}
	Here we see that the effective parameter has to be known for every macroscopic quadrature point. Thus, the sampling domains in the HMM are always centered around such a point. For convenience, in the rest of this section we abbreviate $\bar{x}=x_{K}^{\quadratureindex}$. 
	\begin{remark}[Knowledge about periodicity]\label{rem:Knowledge about periodicity}
		The definition of the parameters involves the solution of cell problems and unit cells. In this work we make the assumption that we know the period length $\delta$ exactly. In a more general setting, without this knowledge, we would introduce an oversampling parameter. However, the analysis of the resulting modeling errors is the task of future research.
	\end{remark}
	In the view of \cref{rem:Knowledge about periodicity} we introduce a triangulation $\mcal{T}_{h}$ of the sampling domain $\samplingdomain{\delta}{\bar{x}}$ and from now on we drop the dependence on the macroscopic variable. As function space, we choose Lagrange finite elements of degree $k\in\N$ with periodic boundary conditions, denoted as $\spaceVh$. Accordingly, we introduce the discrete counterparts of the cell problems and effective parameters. The discrete corrector $\wlMh(\bar{x},\cdot) \in\spaceVh$ for $\secondindex = 1,\dots,\dimMaxwell$ is the solution of
	\begin{align}
	\int_{\samplingdomainname{\delta}} \vec{M}\left(\bar{x},\frac{y}{\delta}\right) \left(\el + \grady \wlMh(\bar{x} ,y )\right) \cdot \grady v^{h}(y) \D y = 0\quad \text{for all }v^{h}\in\spaceVh \eqdot \label{def:cell_prob_wMh}
	\end{align}
	The HMM parameter $\vec{M}^{\HMM}$ (we indicate HMM quantities by a straight H) is given as
	\begin{align*}
	(\vec{M}^{\HMM}(\bar{x}))_{\firstindex,\secondindex} \coloneqq \fint_{\samplingdomainname{\delta}} \vec{M}\left(\bar{x},\frac{y}{\delta}\right) \left(\el + \nabla_{y} \wlMh(\bar{x},y) \right)\cdot \left(\ek + \nabla_{y} \wkMh(\bar{x},y) \right) \D y \eqdot
	\end{align*}
	Consequently, we define the HMM bilinear form using these discrete quantities as
	\begin{align}\label{def:bilinearform_m_HMM}
	\bilinearform{m^{\HMM}}{\vec{\Phi}_{H}}{\vec{\Psi}_{H}} = \bilinearform{}{\vec{M}^{\HMM} \vec{\Phi}_{H}}{ \vec{\Psi}_{H}}_{H} = \sum\nolimits_{K\in\mcal{T}_{H}}\sum\nolimits_{\quadratureindex = 1}^{Q_{K}} \bar{\weight} \vec{M}^{\HMM}(\bar{x}) \vec{\Phi}_{H}(\bar{x}) \cdot\vec{\Psi}_{H}(\bar{x}) \eqdot
	\end{align}
	For the damping parameter we proceed as before and use the discrete corrector $\wlMh$, $\secondindex = 1,\dots,\dimMaxwell$ solution of \cref{def:cell_prob_wMh} and define the HMM damping parameter as
	\begin{align*}
	\left(\vec{R}^{\HMM}\left(\bar{x}\right)\right)_{\firstindex,\secondindex} = \fint_{\samplingdomainname{\delta}} \vec{R}\left(\bar{x},\frac{y}{\delta}\right)\left(\el + \grady \wlMh\left(\bar{x},y\right) \right)\cdot \left(\ek + \grady \wkMh\left(\bar{x},y\right) \right)\D y \eqdot
	\end{align*}
	Thus, the HMM bilinear form related to \cref{def:bilinearform_r_eff} is given as
	\begin{align}\label{def:bilinearform_r_HMM}
	\bilinearform{r^{\HMM}}{\vec{\Phi}_{H}}{\vec{\Psi}_{H}} \coloneqq \sum\nolimits_{K\in\mcal{T}_{H}}\sum\nolimits_{\quadratureindex = 1}^{Q_{K}} \bar{\weight} \vec{R}^{\HMM}\left(\bar{x}\right) \vec{\Phi}_{H}(\bar{x}) \cdot \vec{\Psi}_{H}(\bar{x}) \eqdot
	\end{align}
	For the effective convolution kernel recall the definition in \cref{def:G_eff_reformulated_transformed}. To define the HMM counterpart of \cref{def:bilinearform_g_eff} we need to introduce the discrete cell problems. Therefore, we introduce the corrector $\wlGh$ as the finite element approximation of $\wlG$, i.e., it solves: Find $\wlGh(\cdot,\bar{x},\cdot):[0,T]\to\spaceVh$ such that
	\begin{align*}
	\int_{\samplingdomainname{\delta}} \left[\vec{M}\left(\bar{x},\frac{y}{\delta}\right) \partial_t \grady\wlGh(t,\bar{x},y) + \vec{R}\left(\bar{x},\frac{y}{\delta}\right) \grady\wlGh(t,\bar{x},y) \right]\cdot \nabla_y v^{h}(y)\D y =0 \eqcomma
	\end{align*}
	and 
	\begin{align*}
	\int_{\samplingdomainname{\delta}} \left[\vec{M}\left(\bar{x},\frac{y}{\delta}\right) \grady \wlGh(0,\bar{x},y) + \vec{R}\left(\bar{x},\frac{y}{\delta}\right) \left(\el + \grady \wlMh(\bar{x},y)\right)\right] \cdot \grady v^{h}(y)\D y = 0 \eqcomma
	\end{align*}
	for all $v^{h}\in\spaceVh$. Eventually, the HMM parameter is defined as
	\begin{align}
	\left(\vec{G}^{\HMM}\left(t,\bar{x}\right)\right)_{\firstindex,\secondindex} = \fint_{\samplingdomainname{\delta}} \vec{R}\left(\bar{x},\frac{y}{\delta}\right) \grady \wlGh\left(t,\bar{x},y\right) \cdot \left(\ek + \grady \wkMh\left(\bar{x},y\right) \right)\D y \eqcomma \label{def:G_HMM_symmetric}
	\end{align}
	and the HMM bilinear form is given as
	\begin{align}\label{def:bilinearform_g_HMM}
	\bilinearform{g^{\HMM}}{t;\vec{\Phi}_{H}}{\vec{\Psi}_{H}} \coloneqq \sum\nolimits_{K\in\mcal{T}_{H}}\sum\nolimits_{\quadratureindex = 1}^{Q_{K}} \bar{\weight} \vec{G}^{\HMM}\left(t,\bar{x}\right) \vec{\Phi}_{H}(\bar{x}) \cdot \vec{\Psi}_{H}(\bar{x}) \eqdot
	\end{align}
	The final expression of the system \cref{eq:bilinear_discrete_effective_Maxwell} that includes an effective parameter is the extra source. Similar to the previous considerations, we define the HMM extra source
	\begin{align}
	\left(\vec{J}^{\HMM}\left(t,\bar{x}\right)\right)_{\firstindex,\secondindex} = \fint_{\samplingdomainname{\delta}} \vec{R}\left(\bar{x},\frac{y}{\delta}\right) \grady \wlNh\left(t,\bar{x},y\right) \cdot \left(\ek + \grady \wkMh\left(\bar{x},y\right) \right)\D y \eqcomma \label{def:J_HMM}
	\end{align}
	where the corrector $\wlNh$ is the discrete approximation of $\wlN$.
	\begin{remark}
		For the implementation, a different formulation of the HMM bilinear forms should be used, see \cite[Section 3.3]{Hochbruck2018}, \cite[Section 5.2.1]{Freese2021} that is suited for parallel assembly of finite element matrices.
	\end{remark}

	\subsubsection{The FE-HMM Maxwell system}
	Combining the results we get the Finite Element Heterogeneous Multiscale Method for the general Maxwell system. We search for the HMM solution $\HMMsolutionu:[0,T]\to \spaceVH$ such that
	\begin{align}
	\begin{split}\label{eq:bilinear_HMM_Maxwell}
	\bilinearform{m^{\HMM}}{\partial_{t}\HMMsolutionu(t)}{\vec{\Phi}_{H}} &+ \bilinearform{r^{\HMM}}{\HMMsolutionu(t)}{\vec{\Phi}_{H}} + \int_{0}^{t}\bilinearform{g^{\HMM}}{t-s;\HMMsolutionu(s)}{\vec{\Phi}_{H}} \D s + \bilinearform{a_{H}}{\HMMsolutionu(t)}{\vec{\Phi}_{H}}
	\\
	&= \bilinearform{m^{\HMM}}{\vec{f}^{\HMM}(t)}{\vec{\Phi}_{H}} - \bilinearform{}{\vec{J}^{\HMM}(t)\solutionu_{0,H}}{\vec{\Phi}_{H}}_{H} \quad \text{for all }\vec{\Phi}_{H}\in\spaceVH \eqcomma
	\end{split}\\
	\HMMsolutionu(0) &= \solutionu_{0,H} \eqcomma
	\end{align}
	where $\vec{f}^{\HMM}$ is defined such that
	\begin{align*}
	\bilinearform{m^{\HMM}}{\vec{f}^{\HMM}(t)}{\vec{\Phi}_{H}} = \bilinearform{}{\vec{g}_{H}(t)}{\vec{\Phi}_{H}}_{H} \quad \text{for all }\vec{\Phi}_{H} \in \spaceVH \eqdot
	\end{align*}
	The bilinear forms are given in \cref{def:bilinearform_m_HMM,def:bilinearform_r_HMM,def:bilinearform_g_HMM,def:bilinearform_a_H}. Note that due to the property $\spaceVH\subseteq\spaceVmac$ the discrete Maxwell operator inherits its properties from the continuous one. 
	\begin{lemma}\label{lem:bounded_HMM_parameters}
		The HMM parameters $\vec{M}^{\HMM}$, $\vec{R}^{\HMM}$ are bounded. In addition, $\vec{M}^{\HMM}$ is positive definite with the same constant $\coercivityconstant$ as $\vec{M}^{\delta}$ and $\vec{R}^{\HMM}$ is positive semi-definite.The time-dependent parameters $\vec{G}^{\HMM}$ and $\vec{J}^{\HMM}$ satisfy $\vec{G}^{\HMM}$, $\vec{J}^{\HMM} \in \Lp{\infty}{0,T;\Lp{\infty}{\computationaldomain;\R^{\dimMaxwell \times \dimMaxwell}}}$.
	\end{lemma}
	\begin{proof}
		The proof is as in \cref{lem:bounded_effective_parameters} but uses the discrete cell problems.
	\end{proof}
	Again a direct consequence is the continuity of the HMM bilinear forms as well as the coercivity of $m^{\HMM}$ and the non-negativity of $r^{\HMM}$. The \wellposedness of the FE-HMM system \cref{eq:bilinear_HMM_Maxwell} follows along with the same stability estimate.
	\begin{theorem}\label{thm:wellposedness_HMM}
		Assume that $\delsolutionu,$ $\vec{g}^{\delta}$ and the parameters $\vec{M}^{\delta}$ and $\vec{R}^{\delta}$ satisfy the assumptions from \cref{thm:homogeneous_system}. Then the HMM system \cref{eq:bilinear_HMM_Maxwell} has a unique solution $\HMMsolutionu\in\Wkp{1}{1}{0,T;\spaceVH}$, which satisfies the estimate
		\begin{align}
		\norm{\HMMsolutionu(t)}{\spaceVH} \leq \euler^{\genericconstant_{\vec{G}^{\HMM}}(t)} \left[ \tfrac{1}{\alpha} t \norm{\vec{g}}{\Lp{\infty}{0,t;\spaceVH}} + \left(1+ \tfrac{1}{\alpha} \norm{\vec{J}^{\HMM}}{\Lp{1}{0,t;\Lp{\infty}{\computationaldomain;\R^{\dimMaxwell\times\dimMaxwell}}}}\right) \norm{\solutionu_{0}}{\spaceVH} \right] \eqdot
		\label{eq:stability_Maxwell_integral_discrete}
		\end{align}
		The growth is determined by
		\begin{align}\label{def:C_GH}
		\genericconstant_{\vec{G}^{\HMM}}(t) = \frac{1}{\coercivityconstant} \intt{\norm{\vec{G}^{\HMM}}{\Lp{1}{0,s; \Lp{\infty}{\computationaldomain; \R^{\dimMaxwell\times\dimMaxwell}}}}}{s} \eqdot
		\end{align}
	\end{theorem}
	\begin{proof}
		The proof of unique solvability relies on a reformulation as in \cite[Theorem 5.3.11]{Freese2021} and an application of \cite[Lemma 1.1]{Bossavit2005}. The estimate \cref{eq:stability_Maxwell_integral_discrete} follows thanks to $\spaceVH\subseteq \spaceVmac$ and \cref{ass:quadrature} using the same techniques as in the continuous setting in \cref{eq:stability_Maxwell_integral}.
	\end{proof}

	\section{Semi-discrete error analysis}\label{sec:Semi-discrete error analysis}
	Eventually, we want to estimate the error between the effective solution $\effsolutionu$ of \cref{eq:bilinear_effective_Maxwell} and its HMM approximation $\HMMsolutionu$, solution of \cref{eq:bilinear_HMM_Maxwell}. For that purpose we need to analyze the micro errors, that are the errors between the effective and the HMM parameters. The next result from \cite[Corollary 5.3]{Abdulle2012} or (for $k=1$) \cite[Lemma 3.3]{Abdulle2005} bounds the micro error for the parameter $\vec{M}^{\delta}$.
	
	\begin{lemma}\label{lem:micro_error_M}
		Assume that for $k\in\N$ and every quadrature point $x_{K}^{\quadratureindex}$ we have $\wlM(x_{K}^{\quadratureindex},\cdot) \in \Hk{k+1}{\samplingdomainname{\delta}}$ such that for every $\secondindex =  1,\dots,\dimMaxwell$ and every quadrature point $x_{K}^{\quadratureindex}$ it holds
		\begin{align*}
		\seminorm{\wlM(x_{K}^{\quadratureindex}, \cdot)}{\Hk{k+1}{\samplingdomainname{\delta}}} \leq C \delta^{-k} \sqrt{\seminorm{\samplingdomainname{\delta}}{}} \eqdot
		\end{align*}
		Then, we have the following estimate
		\begin{align*}
		\sup_{K\in\mcal{T}_{H}, \quadratureindex \in\left\lbrace 1,\dots, Q_{K}\right\rbrace} \norm{\vec{M}^{\eff}(x_{K}^{\quadratureindex}) - \vec{M}^{\HMM}(x_{K}^{\quadratureindex})}{F} \leq C \left(\frac{h}{\delta}\right)^{2 k} \eqdot
		\end{align*}	
	\end{lemma}
	A main contribution of this work is to establish similar results for the other parameters arising in the general Maxwell setting.
	\begin{lemma}\label{lem:micro_error_R}
		Assume for $k\in\N$ and for every quadrature point $x_{K}^{\quadratureindex}$ that $\wlM(x_{K}^{\quadratureindex},\cdot)$, $\wlG(0,x_{K}^{\quadratureindex},\cdot) \in\Hk{k+1}{\samplingdomainname{\delta}}$ such that for every $\secondindex =  1,\dots,\dimMaxwell$ and every quadrature point $x_{K}^{\quadratureindex}$ it holds
		\begin{align}\label{eq:regularity_k_assumption_R_micro}
		\seminorm{\wlM(x_{K}^{\quadratureindex},\cdot)}{\Hk{k+1}{\samplingdomainname{\delta}}} , \seminorm{\wlG(0,x_{K}^{\quadratureindex},\cdot)}{\Hk{k+1}{\samplingdomainname{\delta}}} \leq C \delta^{-k} \sqrt{\seminorm{\samplingdomainname{\delta}}{}} \eqdot
		\end{align}
		Then there exists a constant $C>0$ independent of $h$ and $\delta$ such that we have a bound on the Frobenius norm
		\begin{align*}
		\sup_{K\in\mcal{T}_{H}, \quadratureindex \in\left\lbrace 1,\dots, Q_{K}\right\rbrace} \norm{\vec{R}^{\eff}(x_{K}^{\quadratureindex}) - \vec{R}^{\HMM}(x_{K}^{\quadratureindex})}{F} \leq C \left(\frac{h}{\delta}\right)^{2k} \eqdot
		\end{align*}	
	\end{lemma}
	\begin{proof}
		For $\firstindex = 1,\dots,\dimMicro$ we introduce the short notation 
		\begin{align}\label{def:short_notation_micro_lemma}
		\grady \varphi_{\firstindex}^{}(x,y)\coloneqq \left(\ek + \grady \wkM(x,y)\right)\eqcomma \quad \grady \varphi_{\firstindex}^{h}(x,y)\coloneqq \left(\ek + \grady \wkMh(x,y)\right) \eqdot
		\end{align}
		Note that $\varphi_{\firstindex}^{h} - \varphi_{\firstindex}^{}\in\spaceVmic$. We investigate the difference, suppressing the $x_{K}^{\quadratureindex}$ variable. First, add and subtract $\microkappabilinearform{r}{\varphi_{\secondindex}^{h}}{\varphi_{\firstindex}^{}}$. Next, use the transposed cell problem for $\wkG(0)$, cf. \cref{def:cell_prob_wG_0_reformulated_transformed}, tested with $\varphi_{\secondindex}^{} - \varphi_{\secondindex}^{h}$ and the cell problem for $\wlG(0)$ tested with $\varphi_{\firstindex}^{h} - \varphi_{\firstindex}^{}$, i.e,
		\begin{align*}
		\microkappabilinearform{m}{\varphi_{\secondindex}^{}- \varphi_{\secondindex}^{h} }{\wkG(0)} + \microkappabilinearform{r}{\varphi_{\secondindex}^{} - \varphi_{\secondindex}^{h}}{\varphi_{\firstindex}^{}} &= 0 \eqcomma \\ 
		\microkappabilinearform{m}{\wlG(0)}{\varphi_{\firstindex}^{h} - \varphi_{\firstindex}^{}} + \microkappabilinearform{r}{\varphi_{\secondindex}^{}}{\varphi_{\firstindex}^{h} - \varphi_{\firstindex}^{}} &= 0 \eqdot
		\end{align*}
		This yields
		\begin{align*}
		&\left\lvert (\vec{R}^{\eff})_{\firstindex,\secondindex} - (\vec{R}^{\HMM})_{\firstindex,\secondindex} \right\rvert 
		\\&\quad= \frac{1}{\seminorm{\samplingdomainname{\delta}}{}} \Bigl| \microkappabilinearform{r}{\varphi_{\secondindex}^{h} -\varphi_{\secondindex}^{} }{\varphi_{\firstindex}^{} - \varphi_{\firstindex}^{h}} + \microkappabilinearform{m}{\wlG(0)}{\varphi_{\firstindex}^{h} - \varphi_{\firstindex}^{}} - \microkappabilinearform{m}{\varphi_{\secondindex}^{}- \varphi_{\secondindex}^{h} }{\wkG(0)} \Bigr| \eqdot
		\end{align*}
		Finally, we use the standard cell problems \cref{def:cell_prob_wM_reformulated_transformed} and \cref{def:cell_prob_wMh} for $\wlM,$ $\wlMh$ and the transposed ones for $\wkM$, $\wkMh$ tested with $\wkGh(0)$ and $\wlGh(0)$ respectively and add the resulting zeros, i.e.
		\begin{align*}
		\microkappabilinearform{m}{\varphi_{\secondindex}^{}- \varphi_{\secondindex}^{h}}{\wkGh(0)} = 0 \eqcomma \quad
		\microkappabilinearform{m}{\wlGh(0)}{\varphi_{\firstindex}^{h} - \varphi_{\firstindex}^{}} = 0 \eqdot
		\end{align*}
		This leads to
		\begin{align*}
		&\left\lvert (\vec{R}^{\eff})_{\firstindex,\secondindex} - (\vec{R}^{\HMM})_{\firstindex,\secondindex} \right\rvert 
		= \frac{1}{\seminorm{\samplingdomainname{\delta}}{}} \Bigl| \microkappabilinearform{r}{\wlMh - \wlM }{\wkM - \wkMh} \\
		&\qquad+ \microkappabilinearform{m}{\wlG(0) - \wlGh(0)}{\wkMh - \wkM} + \microkappabilinearform{m}{\wlM - \wlMh }{\wkGh(0) - \wkG(0)} \Bigr| \eqdot
		\end{align*}
		An application of the Cauchy-Schwarz inequality together with \cref{def:propertyR} yields
		\begin{align*}
		\left\lvert (\vec{R}^{\eff})_{\firstindex,\secondindex} - (\vec{R}^{\HMM})_{\firstindex,\secondindex} \right\rvert
		\leq C \frac{1}{\seminorm{\samplingdomainname{\delta}}{}} \norm{\wlG(0) - \wlGh(0)}{\spaceVmic} \norm{\wkMh - \wkM}{\spaceVmic} \\
		+ C \frac{1}{\seminorm{\samplingdomainname{\delta}}{}} \norm{\wlMh - \wlM}{\spaceVmic} \left(\norm{\wkM - \wkMh}{\spaceVmic} + \norm{\wkGh(0) - \wkG(0)}{\spaceVmic} \right)
		\end{align*}
		The remaining $\mrm{H}^{1}$-errors are all estimated using \cite[Theorem 3.2.2]{Ciarlet2002}. With the assumption on the regularity of the cell correctors we end up with the asserted estimate.
	\end{proof}

	\subsection{Error estimates for the Sobolev equation and the time-dependent parameter}\label{sec:Error estimates for the Sobolev equation and the time-dependent parameter}
	For the analysis of the micro error of the time-dependent parameters we need an error estimate for the Sobolev equation. The following result is similar to \cite[Theorem 3.2]{Liu2002} but uses a different technique. Therefore, we are able to avoid the application of Gronwall's lemma, which prevents an exponential growth in time.
	\begin{theorem}\label{thm:error_Sobolev_equation}
		Let $\wGh:[0,T]\to\spaceVh$ be solution of 
		\begin{align*}
		\discretemicrobilinearform{m}{\partial_{t}\wGh(t)}{v^{h}} + \discretemicrobilinearform{r}{\wGh(t)}{v^{h}} = 0 \quad \text{for all }v^{h}\in\spaceVh \eqdot \label{eq:Sobolev_equation_variational_homogeneous_discrete}
		\end{align*} 
		Assume that $\wG$, solution of \cref{eq:Sobolev_equation_variational}, satisfies $\wG\in\Ck{1}{0,T;\Hk{k+1}{\samplingdomainname{\delta}}}$. Then there exists a constant $C>0$ independent of $h$ and $t$ such that
		\begin{align*}
		\norm{\wG(t) - \wGh(t)}{\namemicrobilinearform{m}} &\leq C h^{k} \left[ \seminorm{\wG(0)}{\Hk{k+1}{\samplingdomainname{\delta}}} + \seminorm{\wG(t)}{\Hk{k+1}{\samplingdomainname{\delta}}} + \int_{0}^{t}{\seminorm{\wG(s)}{\Hk{k+1}{\samplingdomainname{\delta}}}} \D s\right] \eqdot
		\end{align*}
	\end{theorem}
	\begin{proof}
		The proof is found in \cite[Theorem 5.3.6]{Freese2021} and follows ideas of \cite{Hipp2017}.
	\end{proof}
	With this crucial semi-discrete error estimate we proceed with the analysis of the micro error of the convolution kernel and extra source. 
	\begin{lemma}\label{lem:micro_error_G_first_order}
		Assume that for $k\in\N$ and for every quadrature point $x_{K}^{\quadratureindex}$ we have $\wlM(x_{K}^{\quadratureindex},\cdot) \in \Hk{k+1}{\samplingdomainname{\delta}}$ and $\wlG(t,x_{K}^{\quadratureindex},\cdot) \in \Hk{k+1}{\samplingdomainname{\delta}}$ for all $t\in[0,T]$ such that for every $\secondindex =  1,\dots,\dimMaxwell$ it holds
		\begin{align*}
		\seminorm{\wlM(x_{K}^{\quadratureindex},\cdot)}{\Hk{k+1}{\samplingdomainname{\delta}}} \leq \genericconstant \delta^{-k} \sqrt{\seminorm{\samplingdomainname{\delta}}{}} \eqcomma \quad
		\seminorm{\wlG(0,x_{K}^{\quadratureindex},\cdot)}{\Hk{k+1}{\samplingdomainname{\delta}}} \leq \genericconstant \delta^{-k} \sqrt{\seminorm{\samplingdomainname{\delta}}{}} \eqcomma
		\end{align*}
		for constants $\genericconstant>0$. Then, there exists a constant $\genericconstant>0$ independent of $t,$ $h$ and $\delta$ such that we have a bound on the Frobenius norm
		\begin{align*}
		\sup_{K\in\mcal{T}_{H}, \quadratureindex\in\left\lbrace 1,\dots, Q_{K}\right\rbrace} \norm{\vec{G}^{\eff}(t,x_{K}^{\quadratureindex}) - \vec{G}^{\HMM}(t,x_{K}^{\quadratureindex})}{F} \leq C (1 + t) \left(\frac{h}{\delta}\right)^{k} \eqdot
		\end{align*}
		Under identical assumptions with $\wlG$ replaced by $\wlN$ we get
		\begin{align*}
		\sup_{K\in\mcal{T}_{H}, \quadratureindex\in\left\lbrace 1,\dots, Q_{K}\right\rbrace} \norm{\vec{J}^{\eff}(t,x_{K}^{\quadratureindex}) - \vec{J}^{\HMM}(t,x_{K}^{\quadratureindex})}{F} \leq C (1 + t) \left(\frac{h}{\delta}\right)^{k} \eqcomma
		\end{align*}
		with a constant $\genericconstant>0$ independent of $t$, $h$ and $\delta$.
	\end{lemma}
	\begin{proof}
		We only show the case $k=1$. With the same notation as in \cref{def:short_notation_micro_lemma}, we investigate the difference, suppressing the $x_{K}^{\quadratureindex}$ variable. We use the bilinear forms $\namemicrobilinearform{m}^{\delta}$ and $\namemicrobilinearform{r}^{\delta}$ from \cref{def:bilinearform_m} and \cref{def:bilinearform_r}. The error can be expressed by \cref{def:G_eff_reformulated_transformed} and \cref{def:G_HMM_symmetric} as
		\begin{align*}
		&\left\lvert (\vec{G}^{\eff}(t))_{\firstindex,\secondindex} - (\vec{G}^{\HMM}(t))_{\firstindex,\secondindex} \right\rvert = \frac{1}{\seminorm{\unitcell^{\delta}}{}} \left\lvert \microkappabilinearform{r}{\wlG(t)}{\varphi_{\firstindex}} - \microkappabilinearform{r}{\wlGh(t)}{\varphi_{\firstindex}^{h}} \right\rvert \eqdot
		\end{align*}
		We add and subtract $\microkappabilinearform{r}{\wlG(t)}{\varphi_{\firstindex}^{h}}$ and directly use the boundedness of the parameter $\vec{R}^{\delta}$ as well as the boundedness of the solutions $\wlG(t)$ and $\varphi_{\firstindex}^{h}$ from \cref{lem:wellposedness_wM_wG0_wN0,lem:wellposedness_wG}. With a constant $\genericconstant>0$ this yields 
		\begin{align*}
		&\left\lvert (\vec{G}^{\eff}(t))_{\firstindex,\secondindex} - (\vec{G}^{\HMM}(t))_{\firstindex,\secondindex} \right\rvert \\
		&\quad\leq \frac{\genericconstant}{\seminorm{\unitcell^{\delta}}{}}\norm{\wlG(t)}{\namemicrobilinearform{m}^{\delta}} \norm{\wkM-\wkMh}{\namemicrobilinearform{m}^{\delta}} + \frac{\genericconstant}{\seminorm{\unitcell^{\delta}}{}} \norm{\wlG(t) - \wlGh(t)}{\namemicrobilinearform{m}^{\delta}} \norm{\varphi_{\firstindex}^{h}}{\namemicrobilinearform{m}^{\delta}} \\
		&\quad\leq \frac{\genericconstant}{\sqrt{\seminorm{\unitcell^{\delta}}{}}} \norm{\wkM-\wkMh}{\namemicrobilinearform{m}^{\delta}} + \frac{\genericconstant}{\sqrt{\seminorm{\unitcell^{\delta}}{}}} \norm{\wlG(t) - \wlGh(t)}{\namemicrobilinearform{m}^{\delta}} \eqdot
		\end{align*}
		The final step is to use the finite element error estimate from \cite[Theorem 3.2.2]{Ciarlet2002} in the first expression and the error result from \cref{thm:error_Sobolev_equation} in the second one. At this point, for higher regularity, i.e., $k>1$ we get higher order estimates. This gives
		\begin{align*}
		&\left\lvert (\vec{G}^{\eff}(t))_{\firstindex,\secondindex} - (\vec{G}^{\HMM}(t))_{\firstindex,\secondindex} \right\rvert \leq\frac{\genericconstant}{\sqrt{\seminorm{\unitcell^{\delta}}{}}} h \seminorm{\wkM}{\Hk{2}{\unitcell^{\delta}}} \\
		&\quad+ \frac{\genericconstant}{\sqrt{\seminorm{\unitcell^{\delta}}{}}} h \left(\seminorm{\wlG(0)}{\Hk{2}{\unitcell^{\delta}}} + \seminorm{\wlG(t)}{\Hk{2}{\unitcell^{\delta}}} + \intt{\seminorm{\wlG(s)}{\Hk{2}{\unitcell^{\delta}}}}{s}\right) \eqdot
		\end{align*}
		With the assumption on the correctors, and \cref{thm:H2_Sobolev} we get the final result. The estimate for the extra source follows with the same ideas using the definitions in \cref{def:J_eff_reformulated_transformed,def:J_HMM}.
	\end{proof}

	\begin{remark}
		See \cite[Section 5.3.2]{Freese2021} for an improved estimate in the case $t=0$.
	\end{remark}

	\subsection{Semi-discrete a priori error analysis}\label{sec:Semidiscrete a priori error analysis}
	We analyze the error between the solutions of the effective system \cref{eq:bilinear_effective_Maxwell} and the HMM system \cref{eq:bilinear_HMM_Maxwell}. Recall that we have $\dimMaxwell = 3 (2 + N_{E})$ and assume $\ell\geq 1$. Moreover, we equip the space $\spaceVH$, cf. \cref{def:spaceVH}, with the inner product $\bilinearform{}{\phi^{H}}{\psi^{H}}_{\spaceVH} = \bilinearform{m^{\HMM}}{\phi^{H}}{\psi^{H}}$ and abbreviate $\spaceX \coloneqq \Lp{2}{\Omega;\R^\dimMaxwell}$ which is equipped with the inner product $ \bilinearform{}{\cdot}{\cdot}_{\spaceX} = \bilinearform{m^{\eff}}{\cdot}{\cdot} $. Finally, let $\mrm{Z} \coloneqq \Hk{\ell+1}{\Omega;\R^\dimMaxwell}$ and denote its norm as $ \norm{\phi}{\mrm{Z}} = \norm{\phi}{\Hk{\ell+1}{\computationaldomain;\R^{\dimMaxwell}}} $.
	
	Due to the properties of $\vec{M}^{\delta}$ in \cref{def:propertiesM} and the resulting properties for the bilinear forms $m^{\eff}$ in \cref{lem:bounded_effective_parameters} and $m^{\HMM}$ in \cref{lem:bounded_HMM_parameters} the induced norms are equivalent to the standard $\Lp{2}{\computationaldomain;\R^{\dimMaxwell}}$-norm, i.e., for all $\vec{\Phi}\in\spaceX$ and $\vec{\Phi}_{H}\in\spaceVH$ we get
	\begin{subequations}\label{sys:scalarproducts_equivalent}
		\begin{align}
		&\sqrt{\coercivityconstant} \norm{\vec{\Phi}}{\Lp{2}{\Omega;\R^{\dimMaxwell}}} \leq \norm{\vec{\Phi}}{\spaceX} \leq \sqrt{\boundednessconstant} \norm{\vec{\Phi}}{\Lp{2}{\Omega;\R^{\dimMaxwell}}} \eqcomma \\
		&\sqrt{\coercivityconstant} \norm{\vec{\Phi}_{H}}{\Lp{2}{\Omega;\R^{\dimMaxwell}}} \leq \norm{\vec{\Phi}_{H}}{\spaceVH} \leq \sqrt{\boundednessconstant} \norm{\vec{\Phi}_{H}}{\Lp{2}{\Omega;\R^{\dimMaxwell}}} \eqdot
		\end{align}
	\end{subequations}
	We extend the interpolation operator for N\'{e}d\'{e}lec elements $\mcal{I}_{H}$ from \cref{def:Nedelec_interpolation} to higher dimensions by component-wise application. Additionally, since $\vec{M}^{\HMM}$ is positive definite, we may introduce $\mcal{P}_{H}:\spaceX\to\spaceVH$ such that
	\begin{align}\label{def:projection_macroscopic}
	\bilinearform{m^{\HMM}}{\mcal{P}_{H}\vec{\Phi}}{\vec{\Psi}_{H}} = \bilinearform{m^{\eff}}{\vec{\Phi}}{\vec{\Psi}_{H}} \quad \text{for all }\vec{\Phi}\in\spaceX,\vec{\Psi}_{H}\in\spaceVH \eqdot
	\end{align}
	The following procedure extends the results from \cite{Hipp2017} and \cite{Hochbruck2018} to the more general case considered in this work. We introduce the remainder operator $\Lambda$. This characterizes either the difference between the continuous and discrete Maxwell operator, or for $t\in[0,T]$ and a parameter  $\vec{S} \in \left\lbrace \vec{R},\vec{G}(t),\vec{J}(t) \right\rbrace$ the difference of the effective and HMM representations, i.e.,
	\begin{align*}
	\remainder{\vec{A}}{\vec{A}_{H}} &\coloneqq \mcal{P}_{H}(\vec{M}^{\eff})^{-1} \vec{A} - (\vec{M}^{\HMM})^{-1}\vec{A}_{H} \mcal{I}_{H} \eqcomma \\
	\remainder{\vec{S}^{\eff}}{\vec{S}^{\HMM}} &\coloneqq \mcal{P}_{H}(\vec{M}^{\eff})^{-1} \vec{S}^{\eff} - (\vec{M}^{\HMM})^{-1}\vec{S}^{\HMM} \mcal{I}_{H} \eqdot
	\end{align*}
	Next, we provide a preliminary error estimate.
	\begin{theorem}\label{thm:erroranalysis_01}
		Let $\effsolutionu$ and $\HMMsolutionu$ be the solutions of \cref{eq:bilinear_effective_Maxwell} and \cref{eq:bilinear_HMM_Maxwell}, respectively, and assume $\effsolutionu\in \Ck{1}{[0,T];\mrm{Z}}$. With $\genericconstant_{\vec{G}^{\HMM}}(t)$ given as in \cref{def:C_GH} the error of the semi-discrete HMM-solution is bounded by
		\begin{align}
		\begin{split}\label{eq:erroranalysis_first_bound}
		&\norm{\HMMsolutionu(t) - \effsolutionu(t)}{\spaceX} \leq \euler^{\genericconstant_{\vec{G}^{\HMM}}(t)} \left[\left(1+ \frac{1}{\alpha} \norm{\vec{J}^{\HMM}}{\Lp{1}{0,t;\Lp{\infty}{\computationaldomain;\R^{\dimMaxwell\times\dimMaxwell}}}}\right) \norm{\solutionu_{0,H} - \mcal{I}_{H} \solutionu_{0}}{\spaceVH}\right. \\
		&\quad+ \frac{t}{\alpha} \norm{\vec{f}^{\HMM} - \mcal{P}_{H} \vec{f}}{\Lp{\infty}{0,t;\spaceVH}} + \frac{t}{\alpha} \norm{\remainder{\vec{R}^{\eff}}{\vec{R}^{\HMM}} \effsolutionu}{\Lp{\infty}{0,t;\spaceVH}} \\
		&\quad+ \frac{t}{\alpha} \norm{\remainder{\vec{A}}{\vec{A}_{H}}\effsolutionu}{\Lp{\infty}{0,t;\spaceVH}} + \frac{t}{\alpha} \sup\limits_{s\in[0,t]} \norm{\remainder{\vec{J}^{\eff}(s)}{\vec{J}^{\HMM}(s)} \vec{u}_{0}}{{\spaceVH}} \\
		&\quad+ \frac{t}{\alpha} \sup\limits_{s\in[0,t]} \norm{\int_{0}^{s} \remainder{\vec{G}^{\eff}(s-r)}{\vec{G}^{\HMM}(s-r)} \effsolutionu(r) \D r}{{\spaceVH}} \\ 
		&\quad+\left. \frac{t}{\alpha} \norm{(\mcal{P}_{H}-\mcal{I}_{H}) \partial_{t} \effsolutionu}{\Lp{\infty}{0,t;\spaceVH}} \right] + \norm{(\mcal{I}_{H} - \identityoperator) \effsolutionu(t) }{\spaceX} \eqcomma
		\end{split}
		\end{align}
	\end{theorem}
	\begin{proof}
		We introduce the discrete error $\vec{e}_{H}(t) \coloneqq \HMMsolutionu(t) - \mcal{I}_{H}\effsolutionu(t) \in \spaceVH$. Observe that due to \cref{sys:scalarproducts_equivalent}, we have
		\begin{align}
		\begin{split}\label{eq:erroranalysis_errorsplit}
		\norm{\HMMsolutionu(t) - \effsolutionu(t)}{\spaceX} &\leq \norm{\vec{e}_{H}(t)}{\spaceX} + \norm{(\mcal{I}_{H}-\identityoperator)\effsolutionu(t)}{\spaceX} \\
		&\leq \frac{\sqrt{\boundednessconstant}}{\sqrt{\coercivityconstant}} \norm{\vec{e}_{H}(t)}{\spaceVH} + \norm{ (\mcal{I}_{H} - \identityoperator) \effsolutionu(t)}{\spaceX} \eqdot
		\end{split}
		\end{align}
		We consider the time derivative of the discrete error.
		For any $\vec{\Phi}_{H}\in\spaceVH$ we get
		\begin{align*}\begin{split}
		\bilinearform{m^{\HMM}}{\partial_{t}\vec{e}_{H}(t)}{\vec{\Phi}_{H}} 
		&= \bilinearform{m^{\HMM}}{\partial_{t} \HMMsolutionu(t) - \mcal{P}_{H}\partial_{t} \effsolutionu(t)}{\vec{\Phi}_{H}} + \bilinearform{m^{\HMM}}{\left(\mcal{P}_{H} - \mcal{I}_{H}\right) \partial_{t}\effsolutionu(t)}{\vec{\Phi}_{H}} \eqdot
		\end{split}
		\end{align*}
		A lengthy calculation, where we rewrite the first part on the right-hand side using the effective and HMM systems
		\cref{eq:bilinear_effective_Maxwell} and \cref{eq:bilinear_HMM_Maxwell} as well as \cref{def:projection_macroscopic}, shows that the error $\vec{e}_{H}$ itself satisfies the differential equation
		\begin{align*}
		&\bilinearform{m^{\HMM}}{\partial_{t}\vec{e}_{H}(t)}{\vec{\Phi}_{H}} + \bilinearform{r^{\HMM}}{\vec{e}_{H}(t)}{\vec{\Phi}_{H}} + \int_{0}^{t}\bilinearform{g^{\HMM}}{t-s; \vec{e}_{H}(s)}{\vec{\Phi}_{H}} \D s + \bilinearform{a_{H}}{\vec{e}_{H}(t)}{\vec{\Phi}_{H}} \\
		& \quad= - \bilinearform{}{\vec{J}^{\HMM}(t)\vec{e}_{H}(0)}{\vec{\Phi}_{H}}_{H} + \bilinearform{m^{\HMM}}{\vec{f}^{\HMM}(t)}{\vec{\Phi}_{H}} - \bilinearform{m^{\eff}}{\vec{f}(t)}{\vec{\Phi}_{H}} \\
		&\qquad + \bilinearform{r^{\eff}}{\effsolutionu(t)}{\vec{\Phi}_{H}} - \bilinearform{r^{\HMM}}{\mcal{I}_{H}\effsolutionu(t)}{\vec{\Phi}_{H}} + \bilinearform{a}{\effsolutionu(t)}{\vec{\Phi}_{H}} - \bilinearform{a_{H}}{\mcal{I}_{H}\effsolutionu(t)}{\vec{\Phi}_{H}} \\
		&\qquad+ \int_{0}^{t}\bilinearform{g^{\eff}}{t-s;\effsolutionu(s)}{\vec{\Phi}_{H}} \D s - \int_{0}^{t}\bilinearform{g^{\HMM}}{t-s; \mcal{I}_{H}\effsolutionu(s)}{\vec{\Phi}_{H}} \D s\\
		&\qquad + \bilinearform{m^{\eff}}{(\vec{M}^{\eff})^{-1}\vec{J}^{\eff}(t)\vec{u}_{0}}{\vec{\Phi}_{H}} - \bilinearform{m^{\HMM}}{(\vec{M}^{\HMM})^{-1} \vec{J}^{\HMM}(t)\mcal{I}_{H}\vec{u}_{0}}{\vec{\Phi}_{H}} \\
		&\qquad + \bilinearform{m^{\HMM}}{\left(\mcal{P}_{H} - \mcal{I}_{H}\right) \partial_{t}\effsolutionu(t)}{\vec{\Phi}_{H}} \\
		&\quad= - \bilinearform{}{\vec{J}^{\HMM}(t)\vec{e}_{H}(0)}{\vec{\Phi}_{H}}_{H} + \bilinearform{m^{\HMM}}{\vec{\widetilde{f}}^{\HMM}(t)}{\vec{\Phi}_{H}} \eqcomma
		\end{align*}
		for the right-hand side $\vec{\widetilde{f}}^{\HMM}$ given as
		\begin{align*}
		&\bilinearform{m^{\HMM}}{\vec{\widetilde{f}}^{\HMM}(t)}{\vec{\Phi}_{H}} = \bilinearform{m^{\HMM}}{\vec{f}^{\HMM}(t) - \mcal{P}_{H} \vec{f}(t)}{\vec{\Phi}_{H}} + \bilinearform{m^{\HMM}}{\remainder{\vec{R}^{\eff}}{\vec{R}^{\HMM}} \effsolutionu(t)}{\vec{\Phi}_{H}} \\ 
		&\quad+ \bilinearform{m^{\HMM}}{\remainder{\vec{A}}{\vec{A}_{H}} \effsolutionu(t)}{\vec{\Phi}_{H}} + \bilinearform{m^{\HMM}}{\intt{\remainder{\vec{G}^{\eff}(t-s)}{\vec{G}^{\HMM}(t-s)} \effsolutionu(s)}{s}}{\vec{\Phi}_{H}}\\
		&\quad+ \bilinearform{m^{\HMM}}{\remainder{\vec{J}^{\eff}(t)}{\vec{J}^{\HMM}(t)} \vec{u}_{0}}{\vec{\Phi}_{H}} + \bilinearform{m^{\HMM}}{\left(\mcal{P}_{H} - \mcal{I}_{H}\right) \partial_{t}\effsolutionu(t)}{\vec{\Phi}_{H}} \eqdot
		\end{align*}
		At this point we use the stability estimate \cref{eq:stability_Maxwell_integral_discrete} for this discrete system. Note that we need $\vec{\widetilde{f}} \in \Lp{\infty}{0,t;\spaceVH}$, which is the case due to the assumption on $\effsolutionu$ and the properties of the parameters. 
		Along with \cref{eq:erroranalysis_errorsplit} we showed the result.
	\end{proof}
	We proceed by bounding each expression in \cref{eq:erroranalysis_first_bound} separately, which yields conformity errors, that are given as differences in the bilinear forms. Hence, for discrete functions 
	and $t\in[0,T]$ we define
	\begin{align}
	\triangle\bilinearform{m}{\cdot}{\cdot} &\coloneqq \bilinearform{m^{\eff}}{\cdot}{\cdot} - \bilinearform{m^{\HMM}}{\cdot}{\cdot} \eqcomma \quad &&\triangle\bilinearform{a}{\cdot}{\cdot} \coloneqq \bilinearform{a}{\cdot}{\cdot} - \bilinearform{a_{H}}{\cdot}{\cdot} \eqcomma \nonumber  \\
	\triangle\bilinearform{r}{\cdot}{\cdot} &\coloneqq \bilinearform{r^{\eff}}{\cdot}{\cdot} - \bilinearform{r^{\HMM}}{\cdot}{\cdot} \eqcomma \quad &&
	\triangle\bilinearform{j}{t;\cdot}{\cdot} \coloneqq \bilinearform{}{\vec{J}^{\eff}(t)\cdot}{\cdot} - \bilinearform{}{\vec{J}^{\HMM}(t) \cdot}{\cdot}_{H} \eqcomma \nonumber \\
	\triangle\bilinearform{g}{t;\cdot}{\cdot} &\coloneqq \bilinearform{g^{\eff}}{t;\cdot}{\cdot} - \bilinearform{g^{\HMM}}{t;\cdot}{\cdot} \eqdot \label{def:triangle_g}
	\end{align}
	In \cite[Lemma 2.11, Theorem 3.3]{Hipp2017} two expressions from \cref{eq:erroranalysis_first_bound} are already estimated.
	\begin{lemma}[{\cite[Lemma 2.11, Theorem 3.3]{Hipp2017}}]\label{lem:erroranalysis_01}
		Let $\vec{\Phi}\in \mrm{Z}$. Then there exists a constant $\genericconstant>0$ such that
		\begin{align*}
		\norm{(\mcal{P}_{H} - \mcal{I}_{H}) \vec{\Phi}}{\spaceVH} \leq \genericconstant \norm{(\identityoperator- \mcal{I}_{H})\vec{\Phi}}{\spaceX} + \max\nolimits_{\norm{\vec{\Psi}_{H}}{\spaceVH} = 1} \seminorm{\triangle\bilinearform{m}{\mcal{I}_{H} \vec{\Phi}}{\vec{\Psi}_{H}}}{} \eqdot
		\end{align*}
		Moreover, there exists another constant $\genericconstant>0$ such that
		\begin{align*}
		\norm{\remainder{\vec{A}}{\vec{A}_{H}} \vec{\Phi}}{\spaceVH} \leq \genericconstant \norm{(\identityoperator- \mcal{I}_{H}) \vec{\Phi}}{\spaceVmac} + \max\nolimits_{\norm{\vec{\Psi}_{H}}{\spaceVH} = 1} \seminorm{\triangle\bilinearform{a}{\mcal{I}_{H}\vec{\Phi}}{\vec{\Psi}_{H}}}{} \eqdot
		\end{align*}
	\end{lemma}
	Observe that this gives an estimate in terms of an interpolation error and a conformity error. Following the approach in \cite[Theorem 3.3]{Hipp2017} combined with the boundedness of the parameters yields the following results with a similar structure.
	\begin{lemma}\label{lem:erroranalysis_03}
		Let $\vec{\Phi}\in\mrm{Z}$. Then there exists a constant $\genericconstant>0$ such that
		\begin{align*}
		\norm{\remainder{\vec{R}^{\eff}}{\vec{R}^{\HMM}} \vec{\Phi}}{\spaceVH} \leq \genericconstant \norm{(\identityoperator-\mcal{I}_{H}) \vec{\Phi}}{\spaceX} + \max\nolimits_{\norm{\vec{\Psi}_{H}}{\spaceVH} = 1} \seminorm{\triangle\bilinearform{r}{\mcal{I}_{H} \vec{\Phi}}{\vec{\Psi}_{H}}}{\spaceVH} \eqdot
		\end{align*}
	\end{lemma}
	\begin{proof}
		We follow the proof of \cite[Theorem 3.3]{Hipp2017}, which yields
		\begin{align*}
		\norm{\remainder{\vec{R}^{\eff}}{\vec{R}^{\HMM}} \vec{\Phi}}{\spaceVH} = \max\nolimits_{\norm{\vec{\Psi}_{H}}{\spaceVH} = 1} \left(\bilinearform{r^{\eff}}{\left(\identityoperator- \mcal{I}_{H}\right)\vec{\Phi}}{\vec{\Psi}_{H}} + \triangle\bilinearform{r}{\mcal{I}_{H}\vec{\Phi}}{\vec{\Psi}_{H}}\right) \eqdot
		\end{align*}
		The bound \cref{eq:boundedness_bilinearform_r_eff} yields the result.
	\end{proof}
	Next, we turn our focus to the time-dependent parameters and the corresponding forms, where we start with a result concerning the convolution.
	\begin{lemma}\label{lem:erroranalysis_04}
		For all $s\in[0,T]$, $\vec{\Phi}\in \Ck{1}{[0,s];\mrm{Z}}$ there exists a constant $\genericconstant>0$ such that
		\begin{align*}
		&
		\norm{\int_{0}^{s} \remainder{\vec{G}^{\eff}(s-r)}{\vec{G}^{\HMM}(s-r)} \vec{\Phi}(r) \D r}{{\spaceVH}}\\
		&\quad\leq \genericconstant \norm{(\identityoperator-\mcal{I}_{H}) \vec{\Phi}}{\Lp{\infty}{0,s;\spaceX}} + \max\nolimits_{\norm{\vec{\Psi}_{H}}{\spaceVH} = 1} \seminorm{\int_{0}^{s} \triangle\bilinearform{g}{s-r;\mcal{I}_{H}\vec{\Phi}(r)}{\vec{\Psi}_{H}}\D r}{} \eqdot
		\end{align*}
	\end{lemma}
	\begin{proof}
		Following \cite[Theorem 3.3]{Hipp2017} we obtain from \cref{def:projection_macroscopic,def:bilinearform_g_eff,def:bilinearform_g_HMM}
		\begin{align*}
		&
		\norm{\int_{0}^{s} \remainder{\vec{G}^{\eff}(s-r)}{\vec{G}^{\HMM}(s-r)} \vec{\Phi}(r) \D r}{{\spaceVH}}\\
		&=
		\max_{\norm{\vec{\Psi}_{H}}{\spaceVH} = 1} \bilinearform{m^{\HMM}}{\int_{0}^{s} \remainder{\vec{G}^{\eff}(s-r)}{\vec{G}^{\HMM}(s-r)} \vec{\Phi}(r) \D r}{\vec{\Psi}_{H}} \\
		& = 
		\max_{\norm{\vec{\Psi}_{H}}{\spaceVH} = 1} \left(\int\limits_{0}^{s} \bilinearform{g^{\eff}}{s-r;(\identityoperator- \mcal{I}_{H}) \vec{\Phi}(r) }{\vec{\Psi}_{H}} \D r + \int\limits_{0}^{s} \triangle\bilinearform{g}{s-r;\mcal{I}_{H} \vec{\Phi}(r)}{\vec{\Psi}_{H}} \D r\right)
		\end{align*}
		From the $t$-independent bound in \cref{eq:boundedness_bilinearform_g_eff} we obtain the result.
	\end{proof}
	The next lemma gives an estimate for the error that stems from the extra source term. 
	\begin{lemma}\label{lem:erroranalysis_05}
		For all $s\in[0,T]$ and $\vec{\Phi}\in \mrm{Z}$ there is a constant $\genericconstant>0$ such that
		\begin{align*}
		\norm{\remainder{\vec{J}^{\eff}(s)}{\vec{J}^{\HMM}(s)} \vec{\Phi}}{{\spaceVH}} \leq \genericconstant \norm{(\identityoperator- \mcal{I}_{H}) \vec{\Phi}}{\spaceX} + 	\max\nolimits_{\norm{\vec{\Psi}_{H}}{\spaceVH} = 1} \seminorm{\triangle\bilinearform{j}{s;\mcal{I}_{H} \vec{\Phi}}{\vec{\Psi}_{H}}}{} \eqdot
		\end{align*}
	\end{lemma}
	\begin{proof}
		With the boundedness of $\vec{J}^{\eff}(s)$ independent of $s$ from \cref{lem:bounded_effective_parameters} and the same techniques as above we get the proposed estimate.
	\end{proof}
	Using the inequalities from \crefrange{lem:erroranalysis_01}{lem:erroranalysis_05} in \cref{eq:erroranalysis_first_bound} yields an estimate that consists of conformity errors, errors in the data, and interpolation errors. The latter ones can be bounded using the property of the interpolation $\mcal{I}_{H}$ from \cref{thm:Nedelec_interpolation}, i.e.,
	\begin{align}
	\begin{split}\label{eq:interpolation_estimates}
	\norm{(\identityoperator- \mcal{I}_{H}) \solutionu}{\Lp{\infty}{0,t;\spaceVmac}} &\leq C H^{\ell} \seminorm{\solutionu}{\Lp{\infty}{0,t;\mrm{Z}}} \eqcomma \\
	\norm{(\identityoperator- \mcal{I}_{H}) \partial_{t} \solutionu}{\Lp{\infty}{0,t,\spaceX}} &\leq C H^{\ell} \seminorm{\partial_{t} \vec{u}}{\Lp{\infty}{0,t;\mrm{Z}}} \eqdot
	\end{split}
	\end{align}
	Hence, we have to analyze the conformity errors, i.e., the expressions with $\triangle$. First we recall \cite[Lemma 4.1]{Hochbruck2018} where the conformity errors in $m$ and $a$ have been analyzed.
	\begin{lemma}[{\cite[Lemma 4.1]{Hochbruck2018}}]\label{lem:delta_m_delta_a}
		Assume that the corrector $\wM$ satisfies
		\begin{align*}
		\seminorm{\wM(\bar{x},\cdot)}{\Hk{k+1}{\samplingdomainname{\delta}}} \leq \genericconstant \delta^{-k} \sqrt{\seminorm{\samplingdomainname{\delta}}{}}\quad \text{for all quadrature points }\bar{x} \eqcomma
		\end{align*}
		for a constant $\genericconstant>0$ and furthermore
		\begin{align*}
		\vec{M}^{\eff}|_{K}\in\Wkp{\ell+1}{\infty}{K;\R^{\dimMaxwell\times \dimMaxwell}},\quad \norm{\vec{M}^{\eff}}{\Wkp{\ell+1}{\infty}{K}} \leq \genericconstant \eqcomma
		\end{align*}
		for all $K\in\mcal{T}_{H}$ with a different constant $\genericconstant>0$ independent of $\delta$ and $H$. Then, for all $\vec{\Phi}\in\mrm{Z}$ and $\vec{\Psi}_{H} \in \spaceVH$ we get
		\begin{align*}
		\seminorm{\triangle\bilinearform{m}{\vec{\Phi}}{\vec{\Psi}_{H}}}{} \leq \genericconstant \left(H^{\ell} + \left(\frac{h}{\delta}\right)^{2k}\right) \norm{\vec{\Phi}}{\mrm{Z}}\norm{\vec{\Psi}_{H}}{\spaceX} \eqcomma
		\quad
		\seminorm{\triangle\bilinearform{a}{\vec{\Phi}_{H}}{\vec{\Psi}_{H}}}{} = 0 \eqdot
		\end{align*}
	\end{lemma}
	The bounds on the remaining conformity errors are obtained similarly to \cref{lem:delta_m_delta_a}.
	\begin{lemma}\label{lem:delta_r_g_J}
		Assume that $\wM$, $\wG(0)$ and $\wN(0)$ satisfy
		\begin{align*}
		\seminorm{\wlM(\bar{x}, \cdot)}{\Hk{k+1}{\samplingdomainname{\delta}}} , \, \seminorm{\wlG(0,\bar{x}, \cdot)}{\Hk{k+1}{\samplingdomainname{\delta}}} , \, \seminorm{\wlN(0, \bar{x}, \cdot)}{\Hk{k+1}{\samplingdomainname{\delta}}} \leq C \delta^{-k} \sqrt{\seminorm{\samplingdomainname{\delta}}{}} 
		\eqcomma
		\end{align*}
		for every quadrature point $\bar{x}$ and a constant $\genericconstant>0$ and furthermore
		\begin{align*}
		\vec{M}^{\eff}|_{K},\vec{R}^{\eff}|_{K}\in\Wkp{\ell+1}{\infty}{K;\R^{\dimMaxwell\times \dimMaxwell}},\quad \norm{\vec{M}^{\eff}}{\Wkp{\ell+1}{\infty}{K}}, \norm{\vec{R}^{\eff}}{\Wkp{\ell+1}{\infty}{K}} \leq \genericconstant \eqcomma
		\end{align*}
		for all $K\in\mcal{T}_{H}$ with a constant $\genericconstant>0$ independent of $\delta$ and $H$. Then, for all $\vec{\Phi}\in\mrm{Z}$, $\vec{\Psi}_{H} \in \spaceVH$ and $t\in[0,T]$ we get 
		\begin{align}
		\seminorm{\triangle\bilinearform{r}{\vec{\Phi}}{\vec{\Psi}_{H}}}{} &\leq \genericconstant \left(H^{\ell} + \left(\frac{h}{\delta}\right)^{2k}\right) \norm{\vec{\Phi}}{\mrm{Z}}\norm{\vec{\Psi}_{H}}{\spaceX} \eqcomma \label{eq:error_trinagle_r} \\
		\seminorm{\triangle\bilinearform{g}{t;\vec{\Phi}}{\vec{\Psi}_{H}}}{} &\leq \genericconstant \left( H^{\ell} + ( 1 + 	t) \left(\frac{h}{\delta}\right)^{k}\right) \norm{\vec{\Phi}}{\mrm{Z}}\norm{\vec{\Psi}_{H}}{\spaceX} \eqcomma \label{eq:error_trinagle_g} \\
		\seminorm{\triangle\bilinearform{j}{t;\vec{\Phi}}{\vec{\Psi}_{H}}}{} &\leq \genericconstant \left( H^{\ell} + (1 + t) \left(\frac{h}{\delta}\right)^{k}\right) \norm{\vec{\Phi}}{\mrm{Z}}\norm{\vec{\Psi}_{H}}{\spaceX} \eqdot \label{eq:error_trinagle_J}
		\end{align}
	\end{lemma}
	\begin{proof}
		The first bound in \cref{eq:error_trinagle_r} follows as in \cite[Lemma 4.1]{Hochbruck2018} using the boundedness of $\vec{R}^{\eff}$ and \cref{lem:micro_error_R}. For the bounds in \cref{eq:error_trinagle_g} and \cref{eq:error_trinagle_J} we again use the proof of \cite[Lemma 4.1]{Hochbruck2018} combined with the uniform estimates from \cref{lem:bounded_effective_parameters,lem:micro_error_G_first_order}.
	\end{proof}
	With all preliminary results at hand, we prove the semi-discrete error estimate.
	\begin{theorem}\label{thm:semidiscrete_error_estimate}
		Let $\effsolutionu$ and $\HMMsolutionu$ be the solutions of \cref{eq:bilinear_HMM_Maxwell,eq:bilinear_effective_Maxwell} respectively and assume that $\effsolutionu\in \Ck{1}{[0,T];\mrm{Z}}$ and $\HMMsolutionu \in\Ck{0}{0,T;\spaceVH}$. Under the assumptions of \cref{lem:delta_r_g_J} with $\genericconstant_{\vec{G}^{\HMM}}(t)$ as in \cref{def:C_GH} the error of the semi-discrete HMM-solution is bounded by
		\begin{align*}
		\begin{split}\label{eq:semidiscrete_error_estimate}
		\norm{\HMMsolutionu(t) - \effsolutionu(t)}{\spaceX} &\leq \genericconstant \euler^{ \genericconstant_{\vec{G}^{\HMM}}(t)} \left(1+ t\right) \Biggl[\norm{\solutionu_{0,H} - \mcal{I}_{H} \solutionu_{0}}{\spaceVH} + \norm{\vec{f}^{\HMM} - \mcal{P}_{H} \vec{f}}{\Lp{\infty}{0,t;\spaceVH}} \\
		&\quad + \left( H^{\ell} + \left(\frac{h}{\delta}\right)^{2k}\right) \left[\norm{\effsolutionu}{\Lp{\infty}{0,t,\mrm{Z}}} + \norm{\partial_{t}\effsolutionu}{\Lp{\infty}{0,t,\mrm{Z}}}\right] \\
		&\quad+ \left(1 + t\right) \left(H^{\ell} + t \left(\frac{h}{\delta}\right)^{k}\right) \norm{\effsolutionu}{\Lp{\infty}{0,t;\mrm{Z}}} \Biggr] \eqcomma
		\end{split}
		\end{align*}
	\end{theorem}
	\begin{proof}
		The starting point is the estimate \cref{eq:erroranalysis_first_bound} combined with \crefrange{lem:erroranalysis_01}{lem:erroranalysis_05}. As a first step, we use the interpolation error estimates from \cref{eq:interpolation_estimates}. 
		Hence, we are left with conformity errors but the arguments $\mcal{I}_{H} \effsolutionu$ and $\mcal{I}_{H} \partial_{t} \effsolutionu$ are not in the space $\mrm{Z}$. In order to apply \cref{lem:delta_m_delta_a,lem:delta_r_g_J}, we follow \cite[Theorem 4.5]{Hochbruck2018}. The procedure is similar for all the remaining expressions. Therefore, we only consider $\triangle g$, defined in \cref{def:triangle_g}. For $r,s\in[0,t]$ and $\vec{\Psi}_{H}\in\spaceVH$ we find
		\begin{align*}
		&\seminorm{\triangle\bilinearform{g}{s;\mcal{I}_{H} \effsolutionu(r)}{\vec{\Psi}_{H}}}{} \leq \seminorm{\bilinearform{g^{\eff}}{s;\mcal{I}_{H} \effsolutionu(r)}{\vec{\Psi}_{H}} -
			\bilinearform{g^{\eff}}{s;\effsolutionu(r)}{\vec{\Psi}_{H}}}{} \\
		&\qquad+ \seminorm{\bilinearform{g^{\eff}}{s;\effsolutionu(r)}{\vec{\Psi}_{H}} -
			\bilinearform{g^{\HMM}}{s;\effsolutionu(r)}{\vec{\Psi}_{H}}}{} \\
		&\qquad+ \seminorm{\bilinearform{g^{\HMM}}{s;\mcal{I}_{H} \effsolutionu(r)}{\vec{\Psi}_{H}} - \bilinearform{g^{\HMM}}{s;\effsolutionu(r)}{\vec{\Psi}_{H}}}{} \\
		&\quad\leq \genericconstant \norm{(\mcal{I}_{H} - \identityoperator) \effsolutionu(r)}{\spaceX} \norm{\vec{\Psi}_{H}}{\spaceX} +\seminorm{\triangle\bilinearform{g}{s; \effsolutionu(r)}{\vec{\Psi}_{H}}}{} \\
		&\qquad + \genericconstant \norm{(\mcal{I}_{H} - \identityoperator) \effsolutionu(r)}{\spaceVH} \norm{\vec{\Psi}_{H}}{\spaceVH} \eqdot
		\end{align*}	
		In the last inequality we used the boundedness of the bilinear forms $g^{\eff}$ and $g^{\HMM}$ given in \cref{eq:boundedness_bilinearform_g_eff,lem:bounded_HMM_parameters}. With this, \cite[Lemma 4.4]{Hochbruck2018} and the properties of the interpolation in \cref{def:Nedelec_interpolation} we get
		\begin{align*}
		&\max_{\norm{\vec{\Psi}_{H}}{\spaceVH} = 1} \seminorm{\int\limits_{0}^{s} \triangle\bilinearform{g}{s-r;\mcal{I}_{H}\effsolutionu(r)}{\vec{\Psi}_{H}}\D r}{} \leq \int\limits_{0}^{s} \genericconstant \norm{(\mcal{I}_{H} - \identityoperator) \effsolutionu(r)}{\spaceVH} \D r \\
		&\qquad+ \int\limits_{0}^{s} \genericconstant \sqrt{\frac{\boundednessconstant}{\coercivityconstant}} \norm{(\mcal{I}_{H} - \identityoperator) \effsolutionu(r)}{\spaceX} \D r+  \max_{\norm{\vec{\Psi}_{H}}{\spaceVH} = 1} \int\limits_{0}^{s} \seminorm{\triangle\bilinearform{g}{s-r; \effsolutionu(r)}{\vec{\Psi}_{H}}}{} \D r\\
		&\quad\leq \int\limits_{0}^{s} \genericconstant H^{\ell} \seminorm{\effsolutionu(r)}{\mrm{Z}} \D r + \max_{\norm{\vec{\Psi}_{H}}{\spaceVH} = 1} \int\limits_{0}^{s} \seminorm{\triangle\bilinearform{g}{s-r; \effsolutionu(r)}{\vec{\Psi}_{H}}}{} \D r \eqdot
		\end{align*}
		Here the result from \cref{eq:error_trinagle_g} is applicable in the last expression. This yields
		\begin{align*}
		&\max_{\norm{\vec{\Psi}_{H}}{\spaceVH} = 1} \seminorm{\int_{0}^{s} \triangle\bilinearform{g}{s-r;\mcal{I}_{H}\effsolutionu(r)}{\vec{\Psi}_{H}}\D r}{} \\
		&\quad \leq \int_{0}^{s} \genericconstant H^{\ell} \seminorm{\effsolutionu(r)}{\mrm{Z}} \D r + \int_{0}^{s} \genericconstant \left( H^{\ell} + ( 1 + s - r) \left(\frac{h}{\delta}\right)^{k}\right) \norm{\effsolutionu(r)}{\mrm{Z}} \D r \eqcomma
		\end{align*}
		and finally
		\begin{align*}
		&\sup\limits_{s\in[0,t]} \max_{\norm{\vec{\Psi}_{H}}{\spaceVH} = 1} \seminorm{\int_{0}^{s} \triangle\bilinearform{g}{s-r;\mcal{I}_{H}\effsolutionu(r)}{\vec{\Psi}_{H}}\D r}{} \\
		&\quad\leq \int_{0}^{t} \left(\genericconstant H^{\ell} + \genericconstant ( 1 + t - r) \left(\frac{h}{\delta}\right)^{k}\right) \D r \norm{\effsolutionu}{\Lp{\infty}{0,t;\mrm{Z}}} \eqdot
		\end{align*}
		With this bound and similar ones for the other conformity errors we get the final semi-discrete error estimate. 
	\end{proof}

	\section{Conclusions}\label{sec:conclusions}
	Within this paper we derived a semi-discrete error estimate for the HMM applied to a general class of dispersive Maxwell systems. We provide bounds for the micro error and proved a rigorous error estimate. Along the examination a crucial $\mrm{H}^{2}$ estimate for the Sobolev equation and the \wellposedness of the microscopic and macroscopic problems has been shown. Numerical experiments, which shall be provided in the future, suggest that an improvement of the rate in \cref{lem:micro_error_G_first_order} seems to be possible. This has to be addressed by forthcoming research. 
	
	\section*{Acknowledgments}
	The author thanks D. Gallistl and C. Wieners for helpful discussions and the CRC 1173 of the Karlsruhe Institute of Technology (KIT), where this work originated, for its support.

	\bibliographystyle{abbrv}
	\bibliography{references}	
\end{document}